\documentclass[3p,times]{elsarticle}
\usepackage{enumerate}

\usepackage{amsmath}

\usepackage{array}
\usepackage{amssymb}
 \usepackage{amsthm}
 \usepackage{amsmath}
\usepackage{enumerate}
\usepackage[figuresright]{rotating}
 \usepackage[titletoc]{appendix}
 \usepackage{appendix}

\newtheorem{definition}{Definition}[section]
\newtheorem{lemma}{Lemma}[section]

\newtheorem{property}{Property}[section]
\newtheorem{theorem}{Theorem}[section]
\newtheorem{example}{Example}[section]

\newcommand{\PreserveBackslash}[1]{\let\temp=\\#1\let\\=\temp}
\newcolumntype{C}[1]{>{\PreserveBackslash\centering}p{#1}}
\newcolumntype{R}[1]{>{\PreserveBackslash\raggedleft}p{#1}}
\newcolumntype{L}[1]{>{\PreserveBackslash\raggedright}p{#1}}

\begin{document}

\begin{frontmatter}




\title{Third order quasi-compact schemes for space tempered fractional diffusion equations \footnote{This work was supported by the National Natural Science Foundation of China under Grant  No. 11271173 and No. 11471150.}
}


\author{Yanyan Yu, Weihua Deng\footnote{Corresponding author. E-mail: dengwh@lzu.edu.cn}, Yujiang Wu, Jing Wu }

\address{School of Mathematics and Statistics, Gansu Key Laboratory of Applied Mathematics and Complex Systems, Lanzhou University, Lanzhou 730000, P.R. China}

\begin{abstract}
Power-law probability density function (PDF) plays a key role in both subdiffusion and L\'{e}vy flights. However, sometimes because of the finite of the lifespan of the particles or the boundedness of the physical space, tempered power-law PDF seems to be a more physical choice and then the tempered fractional operators appear; in fact, the tempered fractional operators can also characterize the transitions among subdiffusion, normal diffusion, and L\'{e}vy flights. This paper focuses on the quasi-compact schemes for space tempered fractional diffusion equations, being much different from the ones for pure fractional derivatives. By using the generation function of the matrix and Weyl's theorem, the stability and convergence of the derived schemes are strictly proved. Some numerical simulations are performed to testify the effectiveness and numerical accuracy of the obtained schemes.


\end{abstract}

\begin{keyword}
Tempered fractional calculus, Compact schemes, Stability, Convergence.




\end{keyword}

\end{frontmatter}
%

\section{Introduction}
The CTRW model, composed of waiting times and jump lengths, is a pillar of statistical physics to characterize the anomalous dynamics. The power-law waiting time distribution is generally used to describe the subdiffusion; and the power-law jump length distribution is applied to L\'{e}vy flights. Based on the corresponding CTRW models, the time, space, or time-space fractional diffusion equations are derived to govern the probability density function (PDF) of the particles \cite{RalfMetzler2000The, CompteAlbert1996Stochastic}. Even though this kind of models find wide applications, sometimes the tempered power-law distribution seems to be a more appropriate choice for some biological processes because of the finite lifespan of the particles or the bounded physical space.

The simplest way to do the tempering is to directly cut the very large jump sizes or very long waiting times. Mantegna and Stanley show that the truncated L\'{e}vy flight ultraslowly converges to a Gaussian \cite{MantegnaRN1994Stochastic}. Replacing the way of truncation in \cite{MantegnaRN1994Stochastic} with exponentially truncating  the L\'{e}vy flight, some analytic results to the problem of convergence of truncated L\'{e}vy flights towards the Gaussian stochastic process are presented \cite{KoponenI1995Analytic}. Exponentially tempering the power-law PDF of waiting times or jump lengths seems to become popular nowadays, since it can bring many technical conveniences \cite{Sabzikar2014Tempered}, e.g., making the tempered stochastic process still be L\'{e}vy process.
For capturing the slow convergence of sub-diffusion to a diffusion limit for passive tracers in heterogeneous media, the model with exponentially tempered power-law waiting time distribution is introduced in \cite{Meerschaert2009Tempered}. By exponentially truncating the L\'{e}vy jump distribution, Cartea and del-Castillo-Negrete propose the partial differential equation (PDE) to describe the solute transport in natural systems \cite{CarteaA2007Fluid}; and the truncation effects in superdiffusive front propagation is discussed in \cite{Del-Castillo-NegreteD.2009Truncation}.

Compared with the tempered fractional PDEs, the finite difference methods for fractional PDEs have been much more well developed, e.g., \cite{meerschaert2006finite, Gao2011A, 2012Crank, Chang-MingChen2010Numerical}. A nature idea is to use the Gr\"{u}nwald-Letnikov formula \cite{I1998} to approximate the Riemann-Liouville fractional derivative. Unfortunately, it is unconditionally unstable for the space fractional derivative. Meerschaert and his partners firstly propose the modified version of the Gr\"{u}nwald-Letnikov formula, i.e., the shift Gr\"{u}nwald formula, to effectively solve the space fractional diffusion equation \cite{tadjeran2006, tadjeran2007second}. More recently, a series of high order schemes, including the compact ones, for space fractional diffusion equation are designed. In \cite{Nasir}, the authors use the superconvergent point to get the second order scheme for the Riemann-Liouville fractional derivative. The second and third order WSGD operators are provided in \cite{Tian2012}, and a third order CWSGD operator is given in \cite{zhou2013quasi}. The related more high order schemes can be seen, e.g., \cite{cheninpress,yu2015t}.

For the numerical solution of the tempered fractional PDEs, based on simulating the trajectories of the particles and then performing their average, the algorithms are provided in \cite{Baeumer2010Tempered} and \cite{GajdaJ2010Fractional}. By directly discretizing the tempered fractional operators, a series of second order schemes are given in \cite{Li1996tempered}. Here we further design the third order quasi-compact schemes for the tempered fractional diffusion equation; and note that not only the designing of the scheme but the proof of its numerical stability and convergence is much different from the the ones of the fractional diffusion equation.  Using Weyl's theorem by decomposing the matrix and the generation function of the matrix, we strictly prove the numerical stability and convergence of the derived schemes. And the extensive numerical experiments are performed to confirm the convergence order of the schemes.

The outline of this paper is as follows. In the next section, together with the left and right shifted Gr\"{u}nwald-Letnikov tempered operator, we develop the third order quasi-compact approximations for the left and right Riemann-Liouville tempered fractional derivatives, respectively. And then we focus on discussing the stability and convergence of implicit difference schemes with third order accuracy in space in Section 3. In Section 4, some numerical experiments are carried out to confirm the reliability of the obtained results. We conclude the paper with some remarks in the last section.

\section{Derivation of the quasi-compact approximations for the tempered fractional derivatives}
 We begin with the definitions of $\alpha$-th order left and right Riemann-Liouville (RL) normalized tempered fractional derivatives
\cite{Sabzikar2014Tempered,CarteaA2007Fluid,Baeumer2010Tempered}.
\begin{definition}
If the function $u(x)$ defined in the finite interval $[a,b]$ is regular enough, then
  for any $\lambda\geq0$ the $\alpha$-th order left and right Riemann-Liouville normalized tempered fractional derivatives are, respectively, defined as
\begin{equation}
_a D_x^{\alpha,\lambda}u(x)=\frac{e^{-\lambda x}}{\Gamma(2-\alpha)}\frac{d^2}{dx^2} \int^{x}_{a}(x-s)^{1-\alpha }e^{\lambda s}u(s) ds-\lambda^\alpha u(x) -\alpha\lambda^{\alpha-1}\frac{d u(x)}{dx}=e^{-\lambda x}\,_a D_x^{\alpha}(e^{\lambda x}u(x))-\lambda^\alpha u(x) -\alpha\lambda^{\alpha-1}\frac{d u(x)}{dx}
\end{equation}
and
\begin{equation}
_x D_{b}^{\alpha,\lambda}u(x)=\frac{ e^{\lambda x}}{\Gamma(2-\alpha)}\frac{d^2}{dx^2} \int^{b}_{x}(s -x)^{1-\alpha }e^{-\lambda s}u(s) ds-\lambda^\alpha u(x) +\alpha\lambda^{\alpha-1}\frac{d u(x)}{dx}=e^{\lambda x}\,_x D_{b}^{\alpha}(e^{-\lambda x}u(x))-\lambda^\alpha u(x) +\alpha\lambda^{\alpha-1}\frac{d u(x)}{dx} ,
\end{equation}
where $ 1<\alpha<2$.
Moreover, if $\lambda=0$, then the  Riemann-Liouville normalized  tempered fractional derivatives $_{a} D_x^{\alpha,\lambda}u(x)$ and $_x D_{b}^{\alpha,\lambda}u(x)$ reduce to the Riemann-Liouville  fractional derivatives $_{a} D_x^{\alpha}u(x)$ and $_x D_{b}^{\alpha}u(x)$, respectively.
\end{definition}
For the convenience of presentation, we denote
\[\,_a D_x^{(\alpha,\lambda)}u(x)= e^{-\lambda x}\,_a D_x^{\alpha}(e^{\lambda x}u(x))\]
  and
\[\,_x D_{b}^{(\alpha,\lambda)}u(x)=e^{\lambda x}\,_x D_{b}^{\alpha}(e^{-\lambda x}u(x)).\]
Then the left and right RL tempered fractional derivatives can be rewritten as
\begin{equation}
_a D_x^{\alpha,\lambda}u(x)=\,_a D_x^{(\alpha,\lambda)}u(x) -\lambda^\alpha u(x) -\alpha\lambda^{\alpha-1}\frac{d u(x)}{dx}
\end{equation}
and
\begin{equation}
_x D_{b}^{\alpha,\lambda}u(x)=\,_x D_{b}^{(\alpha,\lambda)}u(x) -\lambda^\alpha u(x) +\alpha\lambda^{\alpha-1}\frac{d u(x)}{dx} .
\end{equation}

Now we derive the quasi-compact approximations for the derivatives $\,_{-\infty} D_x^{\alpha,\lambda}u(x)$ and $\,_ x D_{+\infty}^{\alpha,\lambda}u(x)$, respectively.
Firstly, if the function $u(x)$ is regular enough, then   $\,_{-\infty} D_x^{(1,\lambda)}u(x)$ and $\,_x D_{+\infty}^{(1,\lambda)}u(x)$  are  equivalent to $\lambda u(x)+\frac{d u(x)}{dx}$ and $\lambda u(x) -\frac{d u(x)}{dx}$, respectively, i.e.,
\begin{equation}
 \begin{array}{l }
\displaystyle
 \,_{-\infty} D_x^{(1,\lambda)}u(x)=e^{-\lambda x}\,_{-\infty} D_x^{1}(e^{\lambda x}u(x))
=e^{-\lambda x} \frac{d}{dx}(e^{\lambda x}u(x))
=e^{-\lambda x}\left(\lambda e^{\lambda x}u(x)+e^{\lambda x}\frac{d u(x)}{dx}\right)
=\lambda u(x)+\frac{d u(x)}{dx},
 \end{array}
\end{equation}
\begin{equation}
 \begin{array}{l }
\displaystyle
 \,_x D_{+\infty}^{(1,\lambda)}u(x)=e^{\lambda x}\,_x D_{+\infty}^{1}(e^{-\lambda x}u(x))
=-1\cdot e^{\lambda x} \frac{d}{dx}(e^{-\lambda x}u(x))
=-e^{\lambda x}\left(-\lambda e^{-\lambda x}u(x)+e^{-\lambda x}\frac{d u(x)}{dx}\right)
= \lambda u(x)-\frac{d u(x)}{dx} .
 \end{array}
\end{equation}
Recalling the definitions of the left and right RL tempered fractional derivatives, we get
\begin{equation}
 \begin{array}{l }
\displaystyle
\,_{-\infty} D_x^{\alpha,\lambda}u(x)=\,_{-\infty} D_x^{(\alpha,\lambda)}u(x) -\lambda^\alpha u(x) -\alpha\lambda^{\alpha-1}\frac{d u(x)}{dx}\\\\
\displaystyle~~~~ ~~~~ ~~~~ ~~~~ ~~~~
=\,_{-\infty} D_x^{(\alpha,\lambda)}u(x) -\alpha\lambda^{\alpha-1}\left(\lambda u(x) +\frac{d u(x)}{dx}\right)+\lambda^\alpha(\alpha-1) u(x)
\end{array}
\end{equation}
and
\begin{equation}
 \begin{array}{l }
\displaystyle
_x D_{+\infty}^{\alpha,\lambda}u(x)=\,_x D_{+\infty}^{(\alpha,\lambda)}u(x) -\lambda^\alpha u(x) +\alpha\lambda^{\alpha-1}\frac{d u(x)}{dx}\\\\
\displaystyle~~~~ ~~~~ ~~~~ ~~~~
=\,_x D_{+\infty}^{(\alpha,\lambda)}u(x) -\alpha\lambda^{\alpha-1}\left( \lambda u(x) -\frac{d u(x)}{dx}\right)+\lambda^\alpha(\alpha-1) u(x).
\end{array}
\end{equation}
Thus the left and right RL tempered fractional derivatives at $x\in \mathbb{R} $ can be rearranged as
\begin{equation}\label{lr11}
\,_{-\infty} D_x^{\alpha,\lambda}u(x)=
\,_{-\infty} D_x^{(\alpha,\lambda)}u(x) -\alpha\lambda^{\alpha-1}\,_{-\infty} D_x^{(1,\lambda)}u(x)+\lambda^\alpha(\alpha-1) u(x)
\end{equation}
and
\begin{equation}\label{lr22}
_x D_{+\infty}^{\alpha,\lambda}u(x)
=\,_x D_{+\infty}^{(\alpha,\lambda)}u(x) -\alpha\lambda^{\alpha-1}\,_x D_{+\infty}^{(1,\lambda)}u(x)+\lambda^\alpha(\alpha-1) u(x),
\end{equation}
which play an important role in the following discussions.
Secondly, in  \cite{Baeumer2010Tempered}, Baeumer et al. give an asymptotic expansion of the derivative $\,_{-\infty} D_x^{(\alpha,\lambda)}u(x)$ in the shift Gr\"{u}nwald difference formula, which is useful in constructing high order approximations.

\begin{lemma} \label{2.1}
 Let $1<\alpha<2$,   $u\in W^{n+\alpha,1}(\mathbb{R})$. For any integer $p$ and   $\lambda\geq0$, if we
define the left and right shifted Gr\"{u}nwald-Letnikov (GL) tempered operators by
\begin{equation} \label{shifd}
\Delta_p^{(\alpha,\lambda)} u(x):= \frac{1}{h^\alpha}\sum\limits_{k=0}^\infty g_k^{(\alpha)}e^{-(k-p)\lambda h}u(x-(k-p)h)
\end{equation}
and
\begin{equation} \label{rshifd}
\Lambda_p^{(\alpha,\lambda)} f(x):= \frac{1}{h^\alpha}\sum\limits_{k=0}^\infty g_k^{(\alpha)}e^{-(k-p)\lambda h}u(x+(k-p)h),
\end{equation}
where $h$ is stepsize,
then there are
\begin{equation} \label{shifa}
\Delta_p^{(\alpha,\lambda)} u(x)=\,_{-\infty}D^{(\alpha,\lambda)}_x u(x)+\sum\limits_{k=1}^{n-1}a_{p}^{\alpha, k} \, _{-\infty}D_x^{(\alpha+k,\lambda)}u(x)h^k+O(h^n)
\end{equation}
and
\begin{equation} \label{rshifa}
\Lambda_p^{(\alpha,\lambda)} u(x)=\,_x D_{ +\infty} ^{(\alpha,\lambda)}u(x)+\sum\limits_{k=1}^{n-1}a_{p}^{\alpha, k} \, _x D_{ +\infty}^{(\alpha+k,\lambda)}u(x)h^k+O(h^n)
\end{equation}
uniformly in $x\in  \mathbb{R} $, where  the weights $g_k^{(\alpha)}= \frac{\Gamma(k-\alpha)}{\Gamma(-\alpha)\Gamma(k+1)}$
and  $a_{p}^{\alpha, k}$ are the coefficients of the power series expansion of the functions  $(1-z)^\alpha $ and  $ \mathfrak{w} _p(z)=\left(\frac{1-e^{-z}}{z}\right)^\alpha e^{pz}$, respectively, and the first three terms of the coefficients of $ \mathfrak{w}_p(z)$ are
 \[a_{p}^{\alpha, 0}=1 , \quad a_{p}^{\alpha, 1}=p-\frac{\alpha}{2} , \quad   a_{p}^{\alpha,2}= \frac{\alpha+3\alpha^2-12\alpha p+12p^2}{24}.  \]
\end{lemma}
According to the definitions of the shifted GL tempered fractional derivatives and Lemma \ref{lemma3}, we construct the following third order quasi-compact approximations for the derivatives $\,_{-\infty} D_x^{(1,\lambda)}u(x)$ and $\,_ x D_{+\infty}^{(1,\lambda)}u(x)$.

%



\begin{theorem}\label{th11}
Suppose  $u\in W^{4,1}(\mathbb{ R})$. Define the difference operators by
 \begin{equation} \label{oper1}
 \begin{array}{l}
\displaystyle
   \,_{L}D ^{( 1,\lambda )}_h   u(x) =  \frac{1}{2}\Delta_0^{(1,\lambda)} u(x)+ \frac{1}{2}\Delta_1^{(1,\lambda)} u(x)
   =  \frac{1}{2h } ( e^{\lambda h} u(x+ h )-e^{- \lambda h}u(x- h))
 \end{array}
 \end{equation}
 and
  \begin{equation} \label{oper2}
 \begin{array}{l}
\displaystyle
   \,_{R}D ^{( 1,\lambda )}_h   u(x) =  \frac{1}{2}\Lambda_0^{(1,\lambda)} u(x)+ \frac{1}{2}\Lambda_1^{(1,\lambda)} u(x)
   =  \frac{1}{2h } ( e^{\lambda h} u(x- h )-e^{- \lambda h}u(x+h)).
 \end{array}
 \end{equation}
Then we have
\begin{equation}\label{1lrc}
\begin{array}{l}
\displaystyle
 \,_{L}D_{h}^{(1,\lambda)}   u(x) = \left(1+ \frac{1}{6}   h^2  \,_{-\infty}D^{( 2,\lambda)}_x  \right) \,_{-\infty}D^{(1,\lambda)}_x u(x)+O(h^3), \\\\
\displaystyle \,_{R}D_{h}^{(1,\lambda)} u(x)=\left(1+ \frac{1}{6} h^2  \,_x D^{(2,\lambda)}_{+\infty} \right)   \,_x D^{(1,\lambda)}_{+\infty} u(x)+O(h^3)
\end{array}
\end{equation}
uniformly for $x\in \mathbb{R} $.
Furthermore, the compact approximations to the derivatives $\,_{-\infty} D_x^{(1,\lambda)}u(x)$ and $\,_x D_{ +\infty}^{(1,\lambda)}u(x)$ are
\begin{equation}
 \,_{L}\mathcal{B} \,_{-\infty}D^{ (1,\lambda )}_x u(x) = \,_{L}D ^{( 1,\lambda) }_h   u(x)+O(h^3)
\end{equation}
and
\begin{equation}\label{th1.2}
  \,_{R}\mathcal{B} \,_x D_{\infty}^{(1,\lambda)}u(x) = \,_{R}D ^{ (1,\lambda) }_h   u(x)+O(h^3),
\end{equation}
respectively, where the compact  difference operators are
\begin{equation*}
  \,_{L} \mathcal{B}u(x)=u(x)+\frac{  h^2}{6}e^{-\lambda x}\delta_x^2(e^{ \lambda x} u(x)) =\frac{1}{6} e^{-\lambda h}u(x-h)+\frac{2}{3} u(x)+\frac{1}{6} e^{ \lambda h}u(x+h) ,
\end{equation*}
\begin{equation*}
  \,_{R} \mathcal{B}u(x)=u(x)+\frac{  h^2}{6}e^{\lambda x}\delta_x^2(e^{- \lambda x} u(x)) =\frac{1}{6} e^{\lambda h}u(x-h)+\frac{2}{3} u(x)+\frac{1}{6} e^{ -\lambda h}u(x+h)
\end{equation*}
and
$\delta_x^2 u(x)=(u(x-h)-2 u(x)+u(x+h))/h^2$.
\end{theorem}
\begin{proof}
By Lemma \ref{2.1}, if we take $\alpha=1$ and $p=0$ and $1$ in (\ref{shifd}) and  (\ref{rshifd}), respectively, then there exist
\begin{equation} \label{shifd01}
\Delta_0^{(1,\lambda)} u(x) = \frac{1}{h } ( u(x )-e^{- \lambda h}u(x- h)),\quad \Delta_{1}^{(1,\lambda)} u(x) = \frac{1}{h } (e^{\lambda h} u(x+ h )-u(x))
\end{equation}
and
\begin{equation} \label{rshifd01}
\Lambda_0^{(1,\lambda)} u(x) = \frac{1}{h } ( u(x )-e^{- \lambda h}u(x+h)),\quad\Lambda_{1}^{(1,\lambda)} u(x) = \frac{1}{h } (e^{\lambda h} u(x- h )-u(x)).
\end{equation}
From (\ref{shifa}) and (\ref{rshifa}), it's easy to  check that
\begin{equation} \label{thd11}
\Delta_p^{(1,\lambda)} u(x)=\,_{-\infty}D^{(1,\lambda)}_x u(x)+\sum\limits_{k=1}^{n-1}a_{p}^{1, k} \, _{-\infty}D_x^{(1+k,\lambda)}u(x)h^k+O(h^n),\quad p=0,1
\end{equation}
and
\begin{equation} \label{thl11}
\Lambda_p^{(1,\lambda)} u(x)=\,_x D_{ +\infty} ^{(1,\lambda)}u(x)+\sum\limits_{k=1}^{n-1}a_{p}^{1, k} \, _x D_{ +\infty}^{(1+k,\lambda)}u(x)h^k+O(h^n),\quad p=0,1
\end{equation}
hold.
Combining  (\ref{oper1}) and (\ref{oper2}) with (\ref{thd11}) and (\ref{thl11}), respectively, we get
\begin{equation}\label{th1.11}
 \begin{array}{l}
\displaystyle
   \,_{L}D ^{ (1,\lambda )}_h   u(x) = \frac{1}{2} \Delta_0^{(1,\lambda)} u(x)+  \frac{1}{2}\Delta_1^{(1,\lambda)} u(x) \\\\
   \displaystyle~~~~ ~~~~ ~~~~ ~~~~~~~~
   =  \,_{-\infty}D^{(1,\lambda)}_x u(x)+ \frac{1}{6}  \,_{-\infty}D^{(3,\lambda)}_x u(x)h^2+O(h^3)\\\\
   \displaystyle~~~~ ~~~~ ~~~~ ~~~~~~~~
    =   \left(1+ \frac{1}{6}   h^2  \,_{-\infty}D^{( 2,\lambda)}_x  \right)\,_{-\infty}D^{(1,\lambda)}_x u(x) +O(h^3)
 \end{array}
 \end{equation}
 and
 \begin{equation}\label{th1.12}
   \,_{R}D ^{ (1,\lambda )}_h   u(x)
    =   \left(1+ \frac{1}{6} h^2 \,_x D^{( 2,\lambda)}_ {+\infty} \right)\,_x D^{(1,\lambda)}_{+\infty} u(x) +O(h^3).
 \end{equation}
 Thus Equations (\ref{1lrc}) hold.

Next we establish the discretizations of the operators $1+ \frac{1}{6} h^2 \,_ {-\infty} D^{( 2,\lambda)}_x$ and $1+ \frac{1}{6} h^2 \,_x D^{( 2,\lambda)}_ {+\infty}$. The difference operator  $\delta_x^2 u(x)=(u(x-h)-2 u(x)+u(x+h))/h^2=\frac{\partial^2 u(x)}{\partial x^2}+O(h^2)$ yields that
\begin{equation*}
 \begin{array}{l}
\displaystyle
 \left(1+ \frac{1}{6}   h^2  \,_{-\infty}D^{( 2,\lambda)}_x  \right)  u(x)
=   u(x)+ \frac{1}{6}   h^2 e^{-\lambda x} \frac{d^2}{dx^2} \left(e^{\lambda x}   u(x)\right)
=   u(x)+ \frac{1}{6}   h^2 e^{-\lambda x} \delta_x^2 \left(e^{\lambda x}   u(x)\right) +O(h^4)
\\\\
\displaystyle ~~~~ ~~~~ ~~~~ ~~~~ ~~~~ ~~~~ ~~~~ ~~~~ ~~~~
=  u(x)+ \frac{1}{6}    e^{-\lambda x}  \left(e^{\lambda (x-h)}    u(x-h)-2e^{\lambda x}    u(x)+
e^{\lambda (x+h)}   u(x+h)\right)+O(h^4)
\\\\
\displaystyle ~~~~ ~~~~ ~~~~ ~~~~ ~~~~ ~~~~ ~~~~ ~~~~ ~~~~
=  \frac{1}{6}     e^{-\lambda h }    u(x-h)+\frac{2}{3}    u(x)+ \frac{1}{6}
e^{\lambda h}    u(x+h)+O(h^4)\\\\
\displaystyle ~~~~ ~~~~ ~~~~ ~~~~ ~~~~ ~~~~ ~~~~ ~~~~ ~~~~
=  \,_{L} \mathcal{B}u(x)  +O(h^4)
\end{array}
\end{equation*}
and
\begin{equation*}
 \begin{array}{l}
\displaystyle
 \left(1+ \frac{1}{6}   h^2  \,_x D^{( 2,\lambda)}_{+\infty}  \right)  u(x)%
=  \frac{1}{6}     e^{\lambda h }    u(x-h)+\frac{2}{3}    u(x)+ \frac{1}{6}
e^{-\lambda h}    u(x+h)+O(h^4)\\\\
\displaystyle ~~~~ ~~~~ ~~~~ ~~~~ ~~~~ ~~~~ ~~~~ ~~~~ ~~~~
=  \,_{L} \mathcal{B}u(x)  +O(h^4)
\end{array}
\end{equation*}
hold.
Then we have
\begin{equation*}
 \begin{array}{l}
\displaystyle
 \left(1+ \frac{1}{6}   h^2  \,_{-\infty}D^{( 2,\lambda)}_x  \right)\,_{-\infty}D^{(1,\lambda)}_x u(x)
=  \frac{1}{6}     e^{-\lambda h }   \,_{-\infty}D^{(1,\lambda)}_x u(x-h)+\frac{2}{3}   \,_{-\infty}D^{(1,\lambda)}_x u(x)+ \frac{1}{6}
e^{\lambda h}   \,_{-\infty}D^{(1,\lambda)}_x u(x+h)+O(h^4)
\\\\
\displaystyle~~~~ ~~~~ ~~~~ ~~~~ ~~~~ ~~~~ ~~~~ ~~~~ ~~~~~~~~ ~~~~ ~~~~~~~~
=  \,_{L} \mathcal{B}\,_{-\infty}D^{(1,\lambda)}_x u(x)  +O(h^4)
\end{array}
\end{equation*}
and
\begin{equation*}
 \begin{array}{l}
\displaystyle
 \left(1+ \frac{1}{6}   h^2  \,_x D^{( 2,\lambda)}_{+\infty}  \right)\,_x D^{(1,\lambda)}_{+\infty} u(x)
=  \,_{R} \mathcal{B}\,_x D^{(1,\lambda)}_{+\infty} u(x)  +O(h^4).
\end{array}
\end{equation*}
Together with (\ref{th1.11}) and (\ref{th1.12}), we obtain
 \begin{equation*}
  \,_{L}\mathcal{B} \,_x D_{\infty}^{(1,\lambda)}u(x) =  \,_{L}D ^{ (1,\lambda )}_h   u(x)+O(h^3)
\end{equation*}
and
 \begin{equation*}
  \,_{R}\mathcal{B} \,_x D_{\infty}^{(1,\lambda)}u(x) =  \,_{R}D ^{ (1,\lambda )}_h   u(x)+O(h^3).
\end{equation*}

 \end{proof}

 Next using the compact difference operators $  \,_{L}\mathcal{B}$ and $  \,_{R}\mathcal{B}$,  we derive the corresponding third order quasi-compact approximations to  the  derivatives $\,_{-\infty} D_x^{(\alpha,\lambda)}u(x)$ and $\,_x D_{+\infty} ^{(\alpha,\lambda)}u(x)$.

\begin{theorem}\label{th222}
Let $1<\alpha<2$,   $u\in W^{3+\alpha,1}(  \mathbb{R}  )$. Define the  difference  operators by
\begin{equation}\label{lcom}
\begin{array}{l}
\displaystyle
 \,_{L}D_{h}^{(\alpha,\lambda)}   u(x)
 =  \mu_{-1}\Delta_{-1}^{(\alpha,\lambda)} u(x)+ \mu_{0}\Delta_{0}^{(\alpha,\lambda)} u(x) + \mu_{1}\Delta_{1}^{(\alpha,\lambda)} u(x),  \\\\
\displaystyle \,_{R}D_{h}^{(\alpha,\lambda)} u(x)=  \mu_{-1}\Lambda_{-1}^{(\alpha,\lambda)} u(x)+ \mu_{0}\Lambda_{0}^{(\alpha,\lambda)} u(x) + \mu_{1}\Lambda_{1}^{(\alpha,\lambda)} u(x),
\end{array}
\end{equation}
respectively, where 
the coefficients satisfy
\begin{equation}\label{coe}
\mu_{-1}=\frac{1}{24} (4 - 7\alpha + 3 \alpha^2),\quad
\mu_{0} =\frac{1}{12}  (8 + \alpha - 3\alpha^2),\quad
\mu_{1}=\frac{1}{24} (4 + 5 \alpha + 3 \alpha^2).
\end{equation}
Then there exist
\begin{equation}\label{lrc}
\begin{array}{l}
\displaystyle
 \,_{L}D_{h}^{(\alpha,\lambda)}   u(x) =  \left(1+ \frac{1}{6}   h^2  \,_{-\infty}D^{( 2,\lambda)}_x  \right) \,_{-\infty}D^{(\alpha,\lambda)}_x u(x)+O(h^3), \\\\
\displaystyle \,_{R}D_{h}^{(\alpha,\lambda)} u(x)= \left(1+ \frac{1}{6}   h^2  \,_xD^{( 2,\lambda)}_{+\infty}  \right)  \,_x D^{(\alpha,\lambda)}_{+\infty} u(x)+O(h^3)
\end{array}
\end{equation}
uniformly for $x\in  \mathbb{R} $.
Furthermore,  the following two quasi-compact  approximations have third order accuracy, i.e.,
\begin{equation}\label{th2.3}
 \,_{L} \mathcal{B} \,_{-\infty}D^{ (\alpha,\lambda )}_x u(x) = \,_{L}D ^{( \alpha,\lambda) }_h   u(x)+O(h^3)
\end{equation}
and
\begin{equation}\label{th2.4}
  \,_{R}\mathcal{B} \,_x D_{\infty}^{(\alpha,\lambda)}u(x) = \,_{R}D ^{ (\alpha,\lambda) }_h   u(x)+O(h^3).
\end{equation}
\end{theorem}
 \begin{proof}
Form Lemma \ref{2.1}, we know
\begin{equation}
\Delta_p^{(\alpha,\lambda)} u(x)=\,_{-\infty}D^{(\alpha,\lambda)}_x u(x)+ a_{p}^{\alpha, 1} \, _{-\infty}D_x^{(\alpha+1,\lambda)}u(x)h^1+ a_{p}^{\alpha, 2} \, _{-\infty}D_x^{(\alpha+2,\lambda)}u(x)h^2+O(h^3).
\end{equation}
By taking $p=1,0$ and $-1$, respectively, we have that
\begin{equation*}
\begin{array}{l}
\displaystyle
 \,_{L}D_{h}^{(\alpha,\lambda)}   u(x)
 =  \mu_{-1}\Delta_{-1}^{(\alpha,\lambda)} u(x)+ \mu_{0}\Delta_{0}^{(\alpha,\lambda)} u(x) + \mu_{1}\Delta_{1}^{(\alpha,\lambda)} u(x)\\\\
\displaystyle~~~~  ~~~~ ~~~~ ~~~~ ~~
=(\mu_{-1}+ \mu_{0}+\mu_{1})\,_{-\infty}D^{(\alpha,\lambda)}_x u(x)+  (\mu_{-1}a_{-1}^{\alpha, 1}+ \mu_{0}a_{0}^{\alpha, 1}+\mu_{1}a_{1}^{\alpha, 1}) \, _{-\infty}D_x^{(\alpha+1,\lambda)}u(x)h^1\\\\
\displaystyle~~~~  ~~~~ ~~~~ ~~~~ ~~~~ ~~ +(\mu_{-1}a_{-1}^{\alpha, 2}+ \mu_{0}a_{0}^{\alpha, 2}+\mu_{1}a_{1}^{\alpha, 2}) \, _{-\infty}D_x^{(\alpha+2,\lambda)}u(x)h^2+O(h^3).
\end{array}
\end{equation*}
In order to obtain the results (\ref{lrc}), the coefficients $\mu_{-1}$, $ \mu_{0}$ and $\mu_{1}$  need to solve
\begin{equation*}
\left\{
\begin{array}{lc}
\displaystyle   \mu_{-1}+ \mu_{0}+\mu_{1}=1, \\\\
\displaystyle   \mu_{-1}a_{-1}^{\alpha, 1}+ \mu_{0}a_{0}^{\alpha, 1}+\mu_{1}a_{1}^{\alpha, 1}=0,\\\\
\displaystyle  \mu_{-1}a_{-1}^{\alpha, 2}+ \mu_{0}a_{0}^{\alpha, 2}+\mu_{1}a_{1}^{\alpha, 2}=\frac{1}{6},
\end{array}
\right.
\end{equation*}
which implies (\ref{coe}). Thus (\ref{lrc}) hold. Similar to the proof of Theorem \ref{th11}, (\ref{th2.3}) and (\ref{th2.4}) can be easily got; here we skip them.
 \end{proof}

Now, let us denote
 \begin{equation} \label{apprll}
 \begin{array}{l}
\displaystyle
  \,_{L}D ^{ \alpha,\lambda }_h   u(x) =  \,_{L}D ^{ (\alpha,\lambda) }_h   u(x)  -\alpha\lambda^{\alpha-1} \,_{L}D ^{ (1,\lambda) }_h   u(x)+ \,_{L}\mathcal{B}\lambda^\alpha(\alpha-1) u(x)  \\\\
   \displaystyle~~~~ ~~~~ ~~~~ ~~~~
   =\frac{1}{h^\alpha} \sum\limits_{k=0}^{\infty} w_k^{(\alpha,\lambda)}u(x-(k-1)h) - \frac{ \alpha\lambda^{\alpha-1}}{2h } ( e^{\lambda h} u(x+ h )-e^{- \lambda h}u(x- h)) +\lambda^\alpha (\alpha-1 )\,_{L}\mathcal{B} u(x)
 \end{array}
 \end{equation}
and
 \begin{equation}\label{appr22}
  \begin{array}{l}
\displaystyle
 \,_{R}D ^{ \alpha,\lambda }_h u(x)= \,_{R}D ^{( \alpha,\lambda )}_h u(x)  -\alpha\lambda^{\alpha-1}\,_{R}D ^{(1,\lambda )}_h u(x) +\lambda^\alpha(\alpha-1)\,_{R}\mathcal{B} u(x)\\\\
   \displaystyle~~~~ ~~~~ ~~~~ ~~~~
    = \frac{1}{h^\alpha}\sum\limits_{k=0}^{\infty} w_k^{(\alpha,\lambda)}u(x+(k-1)h)-\frac{ \alpha\lambda^{\alpha-1}}{2h } ( e^{\lambda h} u(x- h )-e^{- \lambda h}u(x+h)) +\lambda^\alpha (\alpha-1 )\,_{R}\mathcal{B} u(x) .
 \end{array}
 \end{equation}
Then together with (\ref{lr11}), (\ref{lr22}), Theorem \ref{th11} and Theorem  \ref{th222}, we have the third order quasi-compact approximations for the $\alpha$-th order left and right RL tempered fractional derivatives $_{-\infty} D_x^{\alpha,\lambda}u(x)$ and $_x D_{\infty}^{\alpha,\lambda}u(x)$:
\begin{equation} \label{se333}
  \,_{L}\mathcal{B} \,_{-\infty}D^{ \alpha,\lambda }_x u(x) = \,_{L}D ^{ \alpha,\lambda }_h   u(x)+O(h^3)
\end{equation}
and
\begin{equation}\label{se344}
 \,_{R} \mathcal{B} \,_x D_{\infty}^{\alpha,\lambda}u(x) = \,_{R}D ^{ \alpha,\lambda }_h   u(x)+O(h^3).
\end{equation}
If $u(x)$ is defined on $[a,b]$, and $u\in W^{3+\alpha,1}(-\infty, b)$ after zero extension,
then $_a D_x^{\alpha,\lambda}u(x)$ has a third order approximation
\begin{equation}\label{se33}
   \begin{array}{l}
\displaystyle
  \,_{L}\tilde{D} ^{ \alpha,\lambda }_h   u(x)
  =\frac{1}{h^\alpha} \sum\limits_{k=0}^{[\frac{x-a}{h}]} w_k^{(\alpha,\lambda)}u(x-(k-1)h) - \frac{ \alpha\lambda^{\alpha-1}}{2h } ( e^{\lambda h} u(x+ h )-e^{- \lambda h}u(x- h))+\lambda^\alpha (\alpha-1 )\,_{L}\mathcal{B} u(x).
  \end{array}
 \end{equation}
And if $u(x)$ is defined on $[a,b]$, and $u\in W^{3+\alpha,1}(a, +\infty)$ after zero extension, then $_x D_b^{\alpha,\lambda}u(x)$ has a third order approximation
 \begin{equation}\label{se34}
   \begin{array}{l}
\displaystyle
 \,_{R}\tilde{D} ^{ \alpha,\lambda }_h u(x)
 = \frac{1}{h^\alpha}\sum\limits_{k=0}^{[\frac{b-x}{h}]} w_k^{(\alpha,\lambda)}u(x+(k-1)h)-\frac{ \alpha\lambda^{\alpha-1}}{2h } ( e^{\lambda h} u(x- h )-e^{- \lambda h}u(x+h)) +\lambda^\alpha (\alpha-1 )\,_{R}\mathcal{B} u(x) .
 \end{array}
 \end{equation}

\begin{property}\label{wpop}
The formulaes (\ref{lcom}) at the grid points $x_i=a+ih$ are denoted as
\begin{equation*}
\begin{array}{l}
\displaystyle
 \,_{L}D_{h}^{(\alpha,\lambda)}   u(x_i)
  =\frac{1}{h^\alpha}\sum_{k=0}^\infty w_k^{(\alpha,\lambda)}u(x_{i-k+1}),\\\\
\displaystyle \,_{R}D_{h}^{(\alpha,\lambda)} u(x)
=\frac{1}{h^\alpha}\sum_{k=0}^\infty w_k^{(\alpha,\lambda)}u(x_{i+k-1}),
\end{array}
\end{equation*}
where the weights are given as
\begin{equation*}
w_0^{(\alpha,\lambda)}=\mu_1g_0^{(\alpha)}e^{\lambda h},\quad
\, w_1^{(\alpha,\lambda)}=\mu_1g_1^{(\alpha)} +\mu_0g_0^{(\alpha)},\quad
 w_k^{(\alpha,\lambda)}=\left(\mu_1g_k^{(\alpha)}+\mu_0 g_{k-1}^{(\alpha)}+\mu_{-1}g_{k-2}^{(\alpha)}\right)e^{(1-k)\lambda h} ,\quad k\geq2,
\end{equation*}
and the coefficients have the following properties:
 \begin{equation}\label{weq}
 \left\{
\begin{array}{lc}
\displaystyle   w_0^{(\alpha,\lambda)}>0,\quad w_1^{(\alpha,\lambda)}\leq0 ,  \quad w_2^{(\alpha,\lambda)} = \frac{e^{-\lambda h}}{48}(8 - 50 \alpha + \alpha^2 + 14 \alpha^3 + 3 \alpha^4),
 \\\\
\displaystyle w_3^{(\alpha,\lambda)}=- \frac{e^{-2\lambda h}}{144}  \alpha (80 - 86 \alpha - 11 \alpha^2 + 14 \alpha^3 + 3 \alpha^4)  ,\quad w_k^{(\alpha,\lambda)}\geq 0,\, k\geq4,\\\\
\displaystyle    \sum_{k=0}^{+\infty}w_k^{(\alpha,\lambda)} = \left( \mu_1 e^{\lambda h}+   \mu_0 +
    \mu_{-1} e^{-\lambda h} \right)  \left( 1 - e^{-\lambda h}  \right)^\alpha,
\end{array}
\right.
 \end{equation}
 where $\mu_j$, $j=-1,0,1$, are given in (\ref{coe}).
\end{property}

\section{Quasi-compact schemes for the space tempered fractional diffusion equations}
Based on the third order quasi-compact discretizations to the space tempered fractional derivatives,
we develop the implicit schemes of the space tempered fractional diffusion equations with the left RL tempered fractional derivative and right RL tempered fractional derivative, respectively. And the detailed numerical stability and convergence analyses are provided.
\subsection{Quasi-compact scheme of the fractional diffusion equation with the left RL tempered fractional derivative}
Here we consider the following initial boundary value problem
\begin{equation}\label{tequ1}
\left\{
\begin{array}{lll}
\displaystyle \frac{\partial u(x,t)}{\partial t} =K  \, _aD_x^{\alpha,\lambda} u(x,t) +f(x,t),  &(x,t)\in (a,b)\times(0,T], \\\\
\displaystyle  u(x,0)=u_0(x),&x\in[a,b],\\\\
\displaystyle   u(a,t)=0,~~~ u(b,t)=u_b(t),  & t\in[0,T],
\end{array}
\right.
\end{equation}
where $1<\alpha<2$. The diffusion coefficient $K$ is a nonnegative constant.  
 Assume that (\ref{tequ1}) has an unique and sufficiently regular solution.

We take the uniform meshes with the time step size $\tau =T/N$ on the interval $[0,T]$ and the space step size $h=(b-a)/M$
 on the interval $[a,b]$, respectively. Then
 \[\left\{(x_i,t_n)\,|x_i=a+ih,\,\,  i=0,\cdots,M;\,\, t_n=n\tau,\,\, n=0,\cdots,N\right\},\]
where $N, M$ are two positive integers.
For convenience, denote
\[u_i^n=u(x_i,t_n),\quad f^{n+1}_i=f(x_i,t_{n+1 }),\quad \delta_tu^n_i=(u_i^{n+1}- u_i^n)/\tau, \quad 0\leq n\leq N-1.\]
Discretizing the time derivative in (\ref{tequ1}) yields
\begin{equation}\label{spacedis}
\delta_tu^n_i =  K (_aD_x^{\alpha,\lambda} u)_i^{n+1}
+f^{n+1 }_i+O(\tau ).
\end{equation}
Recalling  the third order quasi-compact discretizations (\ref{se33}), we act the invertible operator  $\tau \,_{L}\mathcal{B}$ on both sides of (\ref{spacedis}) and obtain
\begin{equation}\label{spacedissss}
\begin{array}{l }
\displaystyle
\,_{L}\mathcal{B} u_i^{n+1} - \,_{L}\mathcal{B} u_i^{n}
=   K \tau  \,_{L}\tilde{D}_{h}^{\alpha,\lambda} u_i^{n+1}
+ \tau\,_{L}\mathcal{B} f^{n+1 }_i  +\tau R^{n+1 }_i,
\end{array}
\end{equation}
where   \[|R^{n+1 }_i| \leq C_1(\tau +h^3). \]
Separating the time layers and replacing  $\,_{L}\tilde{D}_{h}^{\alpha,\lambda} u_i^{n+1}$ by (\ref{apprll}), we have
\begin{equation}
\,_{L}\mathcal{B} u_i^{n+1}
 - K \tau  \left(\frac{1}{ h^{\alpha }}\sum_{k=0}^{i+1}w_{k}^{(\alpha,\lambda)}u_{i-k+1}^{n+1} - \frac{ \alpha\lambda^{\alpha-1}}{2h } ( e^{\lambda h} u_{i+1}^{n+1} -e^{- \lambda h}u_{i-1}^{n+1})+\lambda^\alpha (\alpha-1 )\,_{L}\mathcal{B}  u_{i }^{n+1}\right) = \,_{L}\mathcal{B} u_i^{n} +\tau\,_{L}\mathcal{B} f_{i}^{n+1 } +\tau R^{n+1 }_i.
\end{equation}
Denoting $U_i^n$ as the numerical approximation of $u_i^n$, we obtain the  quasi-compact scheme for (\ref{tequ1}) as follows
\begin{equation} \label{matappl}
\begin{array}{l}
\displaystyle
\,_{L}\mathcal{B} U_i^{n+1}
 -  K \tau \left(\frac{1}{ h^{\alpha }}\sum_{k=0}^{i+1}w_{k}^{(\alpha,\lambda)}U_{i-k+1}^{n+1} - \frac{ \alpha\lambda^{\alpha-1}}{2h } ( e^{\lambda h} U_{i+1}^{n+1} -e^{- \lambda h}U_{i-1}^{n+1})+\lambda^\alpha (\alpha-1 )\,_{L}\mathcal{B} U_{i }^{n+1} \right)
 = \,_{L}\mathcal{B} U_i^{n}  +\tau \,_{L}\mathcal{B}f_{i}^{n+1 } .
\end{array}
\end{equation}
Then the corresponding matrix form of (\ref{matappl}) can be written as
\begin{equation}\label{tmatrixfourthl}
 ( B_l^{ \alpha } -P^{ \alpha }_l )U^{n+1}=   B^{ \alpha }_l U^{n}+\tau B_l^{ \alpha } F_l^{n+1} +H_l^{n+1},
 \end{equation}
 where  
 $
U^n=(U^n_1,U^n_2,\cdots ,U^n_{M-1})^T,\quad F_l^{n+1 }=(f^{n+1 }_1,f^{n+1 }_2,\cdots ,f^{n+1 }_{M-1})^T
$,

\begin{equation}\label{bdef}
\displaystyle B_l^{ \alpha } =\left(                 
  \begin{array}{ccccc}   
    \frac{2}{3}    &  \frac{1}{6} e^{\lambda h} &     &   &    \\  
   \frac{1}{6} e^{-\lambda h}  & \frac{2}{3} & \frac{1}{6} e^{\lambda h}   &    &     \\  
          &   \ddots  & \ddots &   \ddots &     \\  
          &      & \frac{1}{6} e^{-\lambda h}  & \frac{2}{3} & \frac{1}{6} e^{\lambda h}      \\  
          &    &   & \frac{1}{6} e^{-\lambda h}    &\frac{2}{3}   \\  
  \end{array}
\right),                 
\end{equation}
\begin{equation}\label{mp}
 \begin{array}{l}  
\displaystyle P^{ \alpha  }_l=  K \tau (A^{ \alpha  } - \alpha\lambda^{\alpha-1}C^{ \alpha  }+ \lambda^\alpha (\alpha-1 )B_l^{ \alpha }  )\\\\
\displaystyle~~~~  =  K \tau  \frac{1}{  h^{\alpha }} \left(                 
  \begin{array}{ccccc}   
    w_1^{(\alpha,\lambda)}    & w_0^{(\alpha,\lambda)}   &                 &   &    \\  
   w_2^{(\alpha,\lambda)}    & w_1^{(\alpha,\lambda)} &  w_0^{(\alpha,\lambda)}  &    &     \\  
       \vdots      &     w_2^{(\alpha,\lambda)}   &  w_1^{(\alpha,\lambda)}   &    & \\  
    w_{M-2}^{(\alpha,\lambda)} &  \cdots      &\ddots   & \ddots   & w_0^{(\alpha,\lambda)}    \\  
    w_{M-1}^{(\alpha,\lambda)}&w_{M-2}^{(\alpha,\lambda)} & \cdots  & w_2^{(\alpha,\lambda)}    & w_1^{(\alpha,\lambda)} \\  
  \end{array}
\right)                 
-\frac{  \alpha\lambda^{\alpha-1}}{2h } \left(                 
  \begin{array}{ccccc}   
      0 &     e^{\lambda h} &                 &   &    \\  
      -e^{-\lambda h}   &0 &  \  e^{\lambda h}  &    &     \\  
       &  \ddots    &\ddots   & \ddots   & \\  
              & &- e^{-\lambda h}  & 0  &    e^{\lambda h}\\  
     &  &    & -e^{-\lambda h}    & 0  \\  
  \end{array}
\right)\\\\
\displaystyle~~~~ ~~~~ +K \tau\lambda^\alpha (\alpha-1 )  B_l^{ \alpha },             
  \end{array}
\end{equation}
and
\begin{equation} \label{eq47}     
\begin{array}{l}
\displaystyle H_l^{n+1}=\left(                 
 \begin{array}{c} 
       \frac{1}{6} e^{-\lambda h}    \\  
       0   \\  
       \vdots        \\  
        0      \\  
  \end{array}
\right)
(U^n_0-U^{n+1}_0+\tau f^{n+1 }_0)+ K \tau \left(                 
  \begin{array}{c}   
\frac{1}{ h^{\alpha }}w_2^{(\alpha)} + \frac{  \alpha\lambda^{\alpha-1}}{2h }  e^{-\lambda h} +\frac{1}{6}\lambda^\alpha (\alpha-1 ) e^{-\lambda h} \\  
        w_3^{(\alpha)}   \\  
       \vdots        \\  
        w_{M-1}^{(\alpha) }    \\  
         w_{M}^{(\alpha) }   \\  
  \end{array}
\right) U^{n+1}_0  \\\\
   \displaystyle~~~~ ~~~~  ~~~~ +              
\left(                 
  \begin{array}{c}   
     0    \\  
       \vdots   \\  
       0        \\  
           \frac{1}{6} e^{ \lambda h}       \\  
  \end{array}
\right)
(U^n_M-U^{n+1}_M+\tau f^{n+1 }_M)+  K \tau \left(                 
  \begin{array}{c}   
   0  \\  
      \\  
       \vdots        \\  
     \\  
\frac{1}{ h^{\alpha }}w_ 0^{(\alpha)} - \frac{  \alpha\lambda^{\alpha-1}}{2h }  e^{ \lambda h} +\frac{1}{6}\lambda^\alpha (\alpha-1 ) e^{ \lambda h} \\  
  \end{array}
\right) U^{n+1}_M .
 \end{array}
\end{equation}

\begin{property}\label{ppro}
Let $\frac{ 1}{ K \tau}P^{ \alpha  }_l=\{P_{j,k}\}_{(M-1)\times(M-1)}$.
For $\alpha\in(1,2)$ and $0\leq\lambda h\leq1$, the elements of $\{P_{j,k}\}_{(M-1)\times(M-1)}$ have the following properties:
\begin{enumerate}
  \item $P_{j,j}= \frac{1}{  h^{\alpha }} w_1^{(\alpha,\lambda)}+\frac{2}{3}\lambda^\alpha (\alpha-1 )<0,\quad j=1,\cdots,M-1; $
  \item $P_{j,j+1}= \frac{1}{  h^{\alpha }} w_0^{(\alpha,\lambda)}-\frac{  \alpha\lambda^{\alpha-1}}{2h }  e^{\lambda h} +\frac{ e^{\lambda h} }{6}\lambda^\alpha (\alpha-1 ) ,\quad j=1,\cdots,M-2;$

      $P_{j+1,j}= \frac{1}{  h^{\alpha }} w_2^{(\alpha,\lambda)}+\frac{  \alpha\lambda^{\alpha-1}}{2h }  e^{-\lambda h} +\frac{ e^{-\lambda h} }{6}\lambda^\alpha (\alpha-1 ),\quad j=2,\cdots,M-1;$

      $P_{j,j+1}+P_{k+1,k}>0,\quad j=1,\cdots,M-2 ,k=1,\cdots,M-2;$
  \item  $ P_{j+n,j}=\frac{1}{  h^{\alpha }} w_{n+1}^{(\alpha,\lambda)},\quad n=2,\cdots,M-2,j=n+1,\cdots,M-1;$
  \item 
   $ \sum\limits_{k=0}^{+\infty}w_k^{(\alpha,\lambda)} -\frac{  \alpha(\lambda h)^{\alpha-1}}{2  }(-e^{-\lambda h}+ e^{\lambda h} )
+(\lambda h)^\alpha (\alpha-1 ) (\frac{2}{3}+(e^{-\lambda h}+ e^{\lambda h})/6) \leq0$.

\end{enumerate}
\end{property}

\subsection {Quasi-compact scheme of the fractional diffusion equation with the right RL tempered fractional derivative}
We further consider the initial boundary value problem with the right RL tempered fractional derivative
\begin{equation}\label{tequ2}
\left\{
\begin{array}{lll}
\displaystyle \frac{\partial u(x,t)}{\partial t} = K  \, _xD_b^{\alpha,\lambda} u(x,t) +f(x,t),  &(x,t)\in (a,b)\times(0,T], \\\\
\displaystyle  u(x,0)=u_0(x),&x\in[a,b],\\\\
\displaystyle u(a,t)=u_a(t),~~~  u(b,t)=0,  & t\in[0,T],
\end{array}
\right.
\end{equation}
where $1<\alpha<2$.    
 Assume that the solution of (\ref{tequ2}) is unique and sufficiently regular to guarantee the feasibility of achieving the finite difference scheme and establishing its accuracy.

Recalling the third order quasi-compact discretization (\ref{se34}), we obtain
\begin{equation}\label{spacedissss22}
\begin{array}{l }
\displaystyle
\,_{R}\mathcal{B} u_i^{n+1} -
  K \tau  \,_{R}\tilde{D}_{h}^{\alpha,\lambda} u_i^{n+1}
=\,_{R}\mathcal{B} u_i^{n}
+ \,_{R}\mathcal{B} \tau f^{n+1 }_i  +\tau R^{n+1 }_i,
\end{array}
\end{equation}
where   \[|R^{n+1}_i|\leq C_2(\tau +h^3). \]
Denoting $U_i^n$ as the numerical approximation of $u_i^n$ and replacing  $\,_{R}\tilde{D}_{h}^{\alpha,\lambda} u_i^{n+1}$ by (\ref{appr22}), we obtain the quasi-compact scheme of (\ref{tequ2}) as
\begin{equation} \label{matapp22}
\,_{R}\mathcal{B} U_i^{n+1}
 -  K \tau \left(\frac{1}{ h^{\alpha }}\sum_{k=0}^{M-i+1}w_{k}^{(\alpha,\lambda)}U_{i+k-1}^{n+1} - \frac{ \alpha\lambda^{\alpha-1}}{2h } ( e^{\lambda h} U_{i-1}^{n+1} -e^{- \lambda h}U_{i+1}^{n+1})+\lambda^\alpha (\alpha-1 )\,_{R}\mathcal{B} U_{i }^{n+1}\right) = \,_{R}\mathcal{B} U_i^{n} +\,_{R}\mathcal{B}\tau f_{i}^{n+1}.
\end{equation}
The corresponding matrix form of (\ref{matapp22}) can be written as
\begin{equation}\label{matrixfourthr}
  ( B_r^{ \alpha } -P^{ \alpha }_r )U^{n+1}=  B _r^{ \alpha } U^{n}+\tau B _r^{ \alpha }F_r^{n+1}+H_r^{n+1},
 \end{equation}
 where  $B_r^{ \alpha } =(B_l^{ \alpha } )^T$, $P^{ \alpha }_r=( P^{ \alpha }_l)^T$, $H^{n+1}_r$=flipud$(H^{n+1}_l)$; $B_l^{ \alpha }$ is defined in (\ref{bdef}), $P_l^\alpha$ in (\ref{mp}), and $H_l^{n+1}$ in (\ref{eq47}). We further discuss the properties of $P_l^\alpha$.

\begin{lemma}[\cite{bhatia2009positive}]\label{syy}
Let $A\in \mathbb{R}^{n\times n}$.   And it satisfies $v^T Av \geq0$ for all real nonzero vectors $v$, if and only if its symmetric part $H=\frac{A+A^T}{2}$ is positive definite.
\end{lemma}

\begin{lemma}[\cite{chan2007introduction}]\label{sy111}
 Let $H$ be a Toeplitz matrix with a generating function $f\in \mathbb{C}_{2\pi}$. Let $ \varrho_{ \min}(H)$ and $ \varrho_{ \max}(H)$  denote the smallest and
largest eigenvalues of  $H$, respectively. Then we have
\[ f_{ \min}\leq\varrho_{ \min}(H)\leq\varrho_{ \max}(H)\leq f_{ \max},\]
where $f_{ \min}$ and $f_{ \max} $ denote the minimum and maximum values of $f(x)$, respectively.
In particular, if $f_{ \max}\leq0$ and $f_{\min} \neq f_{ \max}$, then $H $ is negative definite.
\end{lemma}
\begin{lemma}[Weyl's Theorem \cite{chan2007introduction}]\label{weylt}
 Let $A,E\in \mathbb{C} ^{n\times n}$ be Hermitian and the
eigenvalues $\varrho_j(A)$, $\varrho_j(E)$, $\varrho_j(A+E)$ be arranged in an increasing order. Then for
each $k=1,2,\cdots,n$, we have
\[  \varrho_k(A)+\varrho_1(E)\leq\varrho_k(A+E)\leq \varrho_k(A)+\varrho_n(E).\]
\end{lemma}
 \begin{theorem}\label{pdefi}
 When $1<\alpha< 2$ and $\lambda h\leq1$, the matrixes  $ P^{ \alpha }_l$ and $P^{ \alpha }_r$ satisfy $v^T P^{ \alpha }_lv <0$ and $v^T P^{ \alpha }_rv <0$, respectively, for all real nonzero vectors $v$.
 \end{theorem}
\begin{proof}
By Lemma \ref{syy}, we just need to prove that their symmetric part  $\frac{P^{ \alpha }_l+(P^{ \alpha }_l)^T}{2}$ is strictly negative  definite.
Define a symmetry matrix $H=\frac{ 1}{K\tau}\frac{P^{ \alpha }_l+(P^{ \alpha }_l)^T}{2}=\{h_{j,k}\}$.

 Next, we discuss the sign of the elements of matrix $H$.
 From Property \ref{ppro}, we know that the elements in the main diagonal of matrix $H$ are  negative, i.e., $$h_{j,j}<0;$$
 except $h_{j,j+2}$, $h_{j+2,j}$ and $h_{j,j}$, all the other elements of matrix $H$ are nonnegative, i.e., $$h_{j,k}\geq0,\,k\neq j-2,j,j+2;$$
 together with Property \ref{wpop}, $h_{j,j+2}=h_{j+2,j}=\frac{1}{  2h^{\alpha }} w_{3}^{(\alpha,\lambda)}=- \frac{e^{-2\lambda h}}{288h^{\alpha }}  \alpha (80 - 86 \alpha - 11 \alpha^2 + 14 \alpha^3 + 3 \alpha^4) ,\,j=1,\cdots,M-2 $.
  Denote $g(\alpha)=- \frac{1}{288 } (80 - 86 \alpha - 11 \alpha^2 + 14 \alpha^3 + 3 \alpha^4)$.
  We can check that $g(\alpha)=0$ have two simple roots: $\alpha_1=1$ and
  $\alpha_2=\frac{1}{9}(-17 + (6184 -  \sqrt{311901})^{\frac{1}{3}} + (6184 + \sqrt{311901})^{\frac{1}{3}})$.
    Because $g(2)= -\frac{1}{6}<0$, $g(\alpha)\geq 0$ for $\alpha\in(1,\alpha_2]$ and $g(\alpha)< 0$ for $\alpha\in(\alpha_2, 2)$. Then $h_{j,j+2}\geq 0$ for $\alpha\in(1,\alpha_2]$ and $h_{j,j+2}< 0$ for $\alpha\in(\alpha_2, 2)$.
    Now we prove that $H$ is strictly negative definite in both of the two cases.

  When $\alpha\in(1,\alpha_2]$, $$h_{j,j+2}=h_{j+2,j}\geq0,\,j=1,\cdots,M-2.$$
  Then matrix $H$ is a strictly diagonally dominant matrix.
  Combining with the Gerschgorin disk theorem and Property \ref{wpop}, we know that the eigenvalues of matrix $H$ are all negative. So $H$ is strictly negative definite.

 When $\alpha\in(\alpha_2, 2)$, $$h_{j,j+2}=h_{j+2,j}<0,\,j=1,\cdots,M-2.$$
 Let us construct a new symmetric Toeplitz matrix $ H^{+}\in\mathbb{R}^{(M-1)\times (M-1)}$,
\begin{equation*}
 \begin{array}{l}  
\displaystyle  H^{+}= \left(                 
  \begin{array}{ccccccc}   
    h_{a}^{+}   & h_{b}^{+}  &    h_{c}^{+} &   &  & & \\  
   h_{b}^{+}    & h_{a}^{+}  &  h_{b}^{+}   &   h_{c}^{+} &  & &  \\  
      h_{c}^{+} & h_{b}^{+}  & h_{a}^{+}    &  h_{b}^{+}  &   h_{c}^{+}  &  & \\  
      &  \ddots & \ddots     &    \ddots    &   \ddots &  \ddots& \\
      &         &        h_{c}^{+}    &  h_{b}^{+} &  h_{a}^{+} &  h_{b}^{+}& h_{c}^{+} \\
      &         &            &   h_{c}^{+}  &  h_{b}^{+}   &  h_{a}^{+} & h_{b}^{+} \\
      &         &            &              & h_{c}^{+}  &  h_{b}^{+}   &  h_{a}^{+} \\
  \end{array}
\right)
  \end{array}
\end{equation*}
with
\[h_{c}^{+}=-h_{j,j+2}>0 ,\,h_{a}^{+} + 2h_{b}^{+} +  2h_{c}^{+}=0,\]
which should be positive definite and can make $ H^{+}+H$ strictly negative  definite.
In order to make $ H^{+}$ be positive definite, we need
the generation function of $ H^{+}$ to be positive  for any $x\in[-\pi,\pi]$, i.e.,
\[f_{H^{+}}(x)
=
h_{a}^{+}-2 h_{c}^{+} +  2h_{b}^{+} \cos(x)+   4h_{c}^{+}\cos(x)^2\geq0 . \]
Let $y=\cos(x)$. Then the above equation can be rewritten as a quadratic function
\begin{equation}\label{dfdgdg}
  f^{+}(y)=h_{a}^{+}-2 h_{c}^{+} +  2h_{b}^{+} y+   4h_{c}^{+}y^2,
\end{equation}
where  $y\in[-1,1]$, $h_{c}^{+}>0$. 
It's easy to check that the discriminant
$\Delta= (2h_{b}^{+})^2-4(h_{a}^{+}-2 h_{c}^{+}) (4h_{c}^{+})=(h_{a}^{+}+2h_{c}^{+})^2-4(h_{a}^{+}-2 h_{c}^{+}) (4h_{c}^{+})=(h_{a}^{+}-6h_{c}^{+})^2$
 and $f^{+}(1)=0$.
 Then when $\Delta=0$, i.e., $h_{a}^{+}=6h_{c}^{+}$, the function (\ref{dfdgdg}) is nonnegative.
From $h_{a}^{+} + 2h_{b}^{+} +  2h_{c}^{+}=0$, we get
$h_{b}^{+}=-4h_{c}^{+}$.
So, after taking $h_{a}^{+}=6h_{c}^{+}$ and $h_{b}^{+}=-4h_{c}^{+}$, $H^+$ is positive definite.

By some simple calculations, it can be shown that all the elements on the main diagonal of $H^{+}+H$ are negative, and the others  are nonnegative; and $ H^{+}+H$ is a strictly diagonally dominant matrix.
 Combining with the Gerschgorin disk theorem,  the eigenvalues of matrix $ H^{+}+H$ are all negative.
Since $ H^{+}$ is positive definite, $ -H^{+}$ is negative definite. 
As $ H=(H^{+}+H)+(-H^{+})$, together with the Weyl Theorem, the eigenvalues of  $ H$ satisfy
 \[\varrho(H)\leq\max\{\varrho(H^{+}+H)  \}+\max\{\varrho(-H^{+})\}<0.\]
Then the Toeplitz matrix $H$ is strictly negative definite.

From the above, when $\alpha\in(1, 2)$, the matrix $H$ is strictly negative definite, which means $\frac{P^{ \alpha }_l+(P^{ \alpha }_l)^T}{2}$ is also strictly negative definite.
By Lemma \ref{syy}, the matrix  $ P^{ \alpha }_l$   satisfies $v^T P^{ \alpha }_lv <0$   for all real nonzero vectors $v$.
 Since $P^{ \alpha }_r=( P^{ \alpha }_l)^T$, the matrix  $ P^{ \alpha }_r$    satisfies $v^T P^{ \alpha }_rv <0$   for all real nonzero vectors $v$.

\end{proof}

\subsection{Stability and convergence analysis}

 In this subsection, we focus  on the stability and convergence of the numerical schemes   and get that the   schemes   have third  order accuracy in space. 
Define
\[V_h=\left\{u:u=\{u_i\}\, {\rm \, is \, a \, grid \, function \, in}\,\{x=ih\}_{i=1}^{M-1}\,{\rm \, and }\, u_0=u_M=0\right\}.\]
For any $u=\{u_i\}\in V_h$, we use the discrete $L^2$ norm as 
\[\|u\|^2= h\sum_{i=1}^{M-1}u_i^2.\]
Next we probe into some properties  of matrix $B_l^{ \alpha } $.
\begin{lemma}\label{dobb}
Let $B_l^{ \alpha } $  be defined in (\ref{bdef}). Then, for any $h\leq1/\lambda $, $B_l^{ \alpha } $ satisfies that
\begin{equation}\label{mmmx}
   v^T B_l^{ \alpha } v >\frac{1}{12}v^Tv ,
\end{equation}
and
\begin{equation}\label{mmmx2}
   v^T B_l^{ \alpha } v <2v^Tv ,
\end{equation}
for all real nonzero vectors $v$.
\end{lemma}
\begin{proof}
It's easy to check that $\frac{(B_l^{ \alpha } -\frac{1}{12}I)+(B_l^{ \alpha } -\frac{1}{12}I)^T}{2}$ is positive definite, i.e.,
$$\varrho\left(\frac{(B_l^{ \alpha } -\frac{1}{12}I)+(B_l^{ \alpha } -\frac{1}{12}I)^T}{2}\right)_j= \varrho\left(\frac{B_l^{ \alpha } +(B_l^{ \alpha } )^T-\frac{1}{6}I}{2}\right)_j= \frac{7}{12}+ \frac{1}{6}(e^{-\lambda h}+ e^{\lambda h})\cos\left(\frac{j\pi}{M}\right)>0,\,\,\, j=1,\cdots, M-1.$$
According to Lemma \ref{syy}, we get  $v^T (B_l^{ \alpha } -\frac{1}{12}I)v >0$ for all real nonzero vectors $v$, which means
\begin{equation}\label{mmmx}
   v^T B_l^{ \alpha } v >v^T\frac{1}{12}Iv .
\end{equation}
On the other hand, we know that $\frac{(B_l^{ \alpha } -2I)+(B_l^{ \alpha } -2I)^T}{2}$ is negative  definite, i.e.,
$$\varrho\left(\frac{(B_l^{ \alpha } -2I)+(B_l^{ \alpha } -2I)^T}{2}\right)_j= \varrho\left(\frac{B_l^{ \alpha } +(B_l^{ \alpha } )^T-4I}{2}\right)_j= -\frac{4}{3}+ \frac{1}{6}(e^{-\lambda h}+ e^{\lambda h})\cos\left(\frac{j\pi}{M}\right)<0,\,\,\, j=1,\cdots, M-1.$$
Then $v^T (B_l^{ \alpha } -2I)v <0$ for all real nonzero vectors $v$, which means that
\begin{equation}\label{mmmx}
   v^T B_l^{ \alpha } v <2v^TIv .
\end{equation}
\end{proof}
\begin{lemma}[\cite{zhou2013quasi}]\label{ll3}
 Assume that $\{k_n\}$ and $\{p_n\}$ are nonnegative sequences, and the sequence ${\phi_n}$ satisfies
 \[\phi_0\leq g_0,\quad \phi_n\leq g_0+\sum_{l=0}^{n-1}p_l+\sum_{l=0}^{n-1}k_l\phi_l,\,n\geq1,\]
 where $ g_0\geq0$. Then the sequence $\{\phi_n\}$ satisfies
\[  \phi_n\leq \left( g_0+\sum_{l=0}^{n-1}p_l\right)\exp\left(\sum_{l=0}^{n-1}k_l\right),\,n\geq1.\]
\end{lemma}

\begin{theorem}\label{stabp1}
Let $\tilde{U}_j^n$ be the exact solution of (\ref{matappl}), and $ U_j^n $ the numerical solution of (\ref{matappl}) obtained in finite precision arithmetic. Then, when $0<h\leq1/\lambda$, the difference scheme (\ref{matappl}) is stable for all $1<\alpha<2$.
\end{theorem}
\begin{proof}
Denoting $\varepsilon^k_i=\tilde{U}_j^k-U_j^k$ and $\varepsilon^k=(\varepsilon^k_1 ,\varepsilon^k_2,\cdots,\varepsilon^k_{M-1})^T$, from (\ref{matappl}), we obtain
\begin{equation} \label{errsta}
(B_l^{ \alpha } -P^{ \alpha }_l)\varepsilon^{k+1}= B_l^{ \alpha }  \varepsilon^{k},
\end{equation}
where $k=0,2,\cdots,n-1$.
 Multiplying (\ref{errsta}) by $h (\varepsilon^{k+1})^T$, we obtain that
\begin{equation}\label{laea}
h (\varepsilon^{k+1})^T B_l^{ \alpha } \varepsilon^{k+1} =h (\varepsilon^{k+1})^T P^{ \alpha }_l \varepsilon^{k+1}+ h (\varepsilon^{k+1})^T B_l ^{ \alpha } \varepsilon^{k}.
\end{equation}
 By Theorem  \ref{pdefi}, we know that the matrix  $P^{ \alpha }_l$  satisfies $v^T P^{ \alpha }_lv <0$  for all real nonzero vectors $v$. Thus
 \[ (\varepsilon^{k+1})^T P^{ \alpha }_l \varepsilon^{k+1}\leq0.\]
 Then (\ref{laea}) leads to
 \begin{equation*}
h (\varepsilon^{k+1})^T B_l^{ \alpha } \varepsilon^{k+1}
\leq h (\varepsilon^{k+1})^T B_l^{ \alpha }  \varepsilon^{k}
\leq \frac{1}{2}(h (\varepsilon^{k+1})^T B_l^{ \alpha }  \varepsilon^{k+1}+h (\varepsilon^{k})^T B_l^{ \alpha }  \varepsilon^{k}),
\end{equation*}
which implies
\[h (\varepsilon^{k+1})^T B_l^{ \alpha } \varepsilon^{k+1}
\leq   h (\varepsilon^{k})^T B_l^{ \alpha }  \varepsilon^{k} .\]
Denoting $E^k= h (\varepsilon^{k})^T B_l ^{ \alpha } \varepsilon^{k} $, we have
\[E^{k+1}\leq E^{k }.\]
Taking  $k$ from $0$ to $n$ yields
\begin{equation*}
   E^{n+1}\leq E^{n }\leq E^{n-1 }\leq\cdots\leq E^{0 }.
\end{equation*}
Together with Lemma \ref{dobb}, we get
\begin{equation*}
\frac{1}{12} ||\varepsilon^{n+1}||^2 \leq E^{n+1} \leq E^{0 }\leq 2||\varepsilon^{0}||^2.
\end{equation*}
Then
\begin{equation*}
  ||\varepsilon^{n+1}||^2 \leq   24||\varepsilon^{0}||^2.
\end{equation*}
Therefore, the difference scheme (\ref{matappl}) is stable.
\end{proof}
\begin{theorem} \label{convv}
Let $u^n$ be the exact solution of (\ref{tequ1}), and $U^n$ the solution of the given finite difference scheme (\ref{matappl}). Then we have
\[\|u^n-U^n\|\leq C(\tau +h^3),\]
for all $1\leq n\leq N$, where $C$ is a constant independent of $n$, $\tau$, and $h$.
\end{theorem}
\begin{proof}
Denoting $\epsilon^k_j=u(x_j,t_k)-U_j^k$, from (\ref{spacedissss}), we obtain
\begin{equation} \label{errstacon}
(B_l^{ \alpha } -P^{ \alpha }_l)\epsilon^{k+1}= B_l^{ \alpha }  \epsilon^{k}+\tau R^{k+1 },
\end{equation}
where $\epsilon^k=(\epsilon^ K ,\epsilon^k_2,\cdots,\epsilon^k_{M-1})^T$, $R^{k+1 }=(R^{k+1 }_1,R^{k+1 }_2,\cdots,R^{k+1 }_{M-1})^T$ and $0\leq k\leq n-1$.
  Multiplying (\ref{errstacon}) by $h (\epsilon^{k+1})^T$, we get
\begin{equation*}
h (\epsilon^{k+1})^T(B_l^{ \alpha } -P^{ \alpha }_l)\epsilon^{k+1}=h (\epsilon^{k+1})^T B_l^{ \alpha }  \epsilon^{k}+\tau h (\epsilon^{k+1})^T R^{k+1 },
\end{equation*}
  By Theorem  \ref{pdefi}, we have
\[ \frac{1}{2}h (\epsilon^{k+1})^T B_l^{ \alpha }  \epsilon^{k+1}\leq \frac{1}{2}h (\epsilon^{k })^T B_l^{ \alpha }  \epsilon^{k}+\tau h (\epsilon^{k+1})^T R^{k+1 }.\]
Let $e^k= h (\epsilon^{k})^T B_l^{ \alpha }  \epsilon^{k} $, and we have
\begin{equation}\label{qqqq}
  e^{k+1} \leq  e^k + 2\tau h (\epsilon^{k+1})^T R^{k+1 }.
\end{equation}
Together with Lemma \ref{dobb}, summing up (\ref{qqqq}) for all $0\leq k\leq n-1$ shows that
\begin{equation*}
\frac{1}{12} ||\epsilon^{n}||^2 \leq e^{n} \leq e^0 +2\tau \sum_{k=0}^{n-1}  h (\epsilon^{k+1})^T R^{k+1 }
  \leq \frac{1}{48 \tau}2\tau||\epsilon^{n }||^2+ \frac{24 \tau}{2}2\tau||R^{n}||^2+2\tau \sum_{k=0}^{n-2}  h (\epsilon^{k+1})^T R^{k+1 } .
\end{equation*}
Then
\begin{equation*}
 \frac{1}{24}   \| \epsilon^n\|^2
 \leq  24 \tau^2||R^{n}||^2+2\tau \sum_{k=0}^{n-2}  h (\epsilon^{k+1})^T R^{k+1 }
\leq  \frac{1}{24}\tau \sum_{k=1}^{n-1} ||\epsilon ^{k  }||^2+24\tau \sum_{k=1}^{n-1} || R^{k }||^2+ 24 \tau^2||R^{n}||^2.
\end{equation*}
Noticing that $|R^{k+1}_j|\leq c(\tau+h^3)$ for $1\leq k\leq n$ and utilizing the discrete Gronwall's inequality,   we obtain
\begin{equation*}
    \|\epsilon^n\|^2
\leq  \tau \sum_{k=1}^{n-1} ||\epsilon ^{k  }||^2+24^2\tau \sum_{k=1}^{n-1} || R^{k }||^2
  \leq C(\tau +h^3)^2.
\end{equation*}
\end{proof}

By the similar idea, we can prove the following results; and the details are omitted here.

\begin{theorem}\label{stabp}
Let  $\tilde{U}_j^n$ be the exact solution of  (\ref{matapp22}), and $ U_j^n $ the numerical solution of  (\ref{matapp22}) obtained in finite precision arithmetic.
Then when $ 0<h\leq1/\lambda$, the difference schemes (\ref{matapp22}) is stable for all $1<\alpha<2$.
\end{theorem}

\begin{theorem} \label{convv}
Let $u^n$ be the exact solution of (\ref{tequ2}), and $U^n$ the solution of the given finite difference scheme (\ref{matapp22}). Then we have
\[\|u^n-U^n\|\leq C(\tau +h^3),\]
for all $1\leq n\leq N$, where $C$ is a constant independent of $n$, $\tau$, and $h$.
\end{theorem}

\section{Numerical examples}
In this section, we discuss the effectiveness of the third order quasi-compact difference schemes derived in the above. And the presented numerical results confirm the theoretical ones.

Let
\[e(\tau,h)= \left(h\sum_{i=1}^{M-1}\left(u(x_i,t_N)-U_i^ N\right)^2\right)^\frac{1}{2},\]
where $u(x_i,t_N)$  represents the exact solution and $U_i^ N$ the numerical solutions at the grid point $(x_i,t_N)$ with the mesh step sizes $\tau$ and $h$.
Together with the equations
\[  \,_a D_{x}^{(\alpha ,\lambda)}(e^{-\lambda x}(x-a)^{j })=\frac{\Gamma(1+j)e^{-\lambda x}}{\Gamma(1+j-\alpha)}(x-a)^{j-\alpha } \]
and
\[  \,_x D_{b}^{(\alpha ,\lambda)}(e^{ \lambda x}(b-x)^{j })=\frac{\Gamma(1+j)e^{ \lambda x}}{\Gamma(1+j-\alpha)}(b-x)^{j-\alpha }, \]
we show the following examples. To test the order of convergence, except Table \ref{tt5.3}, where $\tau=h^{3/2}$, for all the other Tables, $\tau=h^{3}$ is taken.
\begin{example}\label{tex5.1}
We consider the   tempered space fractional diffusion equation
\begin{equation}
\frac{\partial u}{\partial t}=\,_0D_x^{\alpha,\lambda }u(x)-e^{-t-\lambda x}\left(x^j+\frac{\Gamma(j+1) x^{j -  \alpha}}{\Gamma(1+j - \alpha)}-\alpha\lambda^{\alpha-1} (j x ^{j-1}-\lambda x^j)-\lambda^\alpha x ^j\right) ,\quad (x,t)\in(0,1)\times(0,0.1],
\end{equation}
with the boundary conditions  $u(0,t)=0$,  $u(1,t)=e^{-t-\lambda}$ and  the initial value  $u(x,0)=e^{-\lambda x}x^j, \,x\in[0,1]$, where $j\in \mathbb{N}$. The exact solution is  $u(x)=e^{-t-\lambda x}x^{j}$.
\end{example}

In Table \ref{tex5.1t1}, we confirm the convergence orders and show that the regularity of the solution is necessary for obtaining the desired convergence orders, even though it is weaker than $u\in W^{3+\alpha,1}( \mathbb{R} )$ required in the proof of Theorem \ref{th222}. Table \ref{tex4444} further confirms the convergence orders and shows that, as in the proof of Theorem \ref{stabp1}, the condition $\lambda h<1$ is required for ensuring the stability of the schemes.

\begin{table}[htbp]
\centering\small
\caption{ The errors $e(\tau,h)$ and spatial convergence orders of the quasi-compact scheme (\ref{matappl}) by computing Example \ref{tex5.1} with $\lambda=1$.}\vspace{1em} {\begin{tabular} {@{}cccccccc@{}} \hline
 &  &\multicolumn{2}{c}{$j=1$} & \multicolumn{2}{c}{$j=3$ }& \multicolumn{2}{c}{$j=5$ }\\
\cline{3-4} \cline{5-6} \cline{7-8}
$\alpha$&$h$& $e(\tau,h)$ &rate & $e(\tau,h)$ &rate  & $e(\tau,h)$ &rate \\
 \hline
1.1&$0.1$   & $   4.4768e-05    $ &   $         $   & $      7.2042e-06 $ &   $         $  & $    6.0259e-06  $ &   $         $ \\
   &$0.05$  & $    1.4666e-05   $ &   $   1.6100    $   &  $    8.9495e-07   $ &   $   3.0090    $  &  $    7.6037e-07  $ &   $    2.9864    $ \\
   &$0.025$ & $     3.9039e-06  $ & $      1.9095 $     &  $   1.1183e-07    $  &   $   3.0005  $  &  $   9.5387e-08   $  &   $   2.9948   $   \\
   &$0.0125$& $    7.7769e-07   $ &   $   2.3277    $     &  $    1.3986e-08 $  &   $ 2.9992   $   &  $     1.1945e-08 $  &   $  2.9974    $  \\
 \hline
1.5&$0.1$   & $    2.9214e-04    $ &   $          $   & $   2.2124e-05   $ &   $          $ & $    9.1408e-05   $ &   $          $  \\
   &$0.05$  &  $   6.1994e-05      $ & $   2.2364  $   &  $ 3.4597e-06    $ &   $   2.6769  $  &  $   1.1772e-05   $ &   $ 2.9569    $    \\
   &$0.025$ &  $     1.1896e-05  $  &  $  2.3816   $  &  $  4.9266e-07    $  &  $     2.8120 $  &  $      1.4927e-06  $  &  $  2.9794    $   \\
   &$0.0125$&  $   2.1899e-06      $  &   $   2.4416  $ &  $  6.6604e-08     $  &   $    2.8869   $  &  $  1.8791e-07    $  &   $  2.9898    $   \\
\hline
1.9&$0.1$   & $     6.2875e-04  $ &   $         $         &   $  1.1388e-04     $ &   $         $ &     $     2.7192e-04  $ &   $         $  \\
   &$0.05$  &  $   1.4818e-04     $ &   $   2.0852   $    &  $   1.6833e-05     $ &   $ 2.7582  $  &  $    3.4977e-05    $ &   $    2.9587 $  \\
   &$0.025$ &  $     3.3898e-05   $  &  $   2.1280   $     &  $    2.2847e-06   $  &  $   2.8812  $  &  $    4.4201e-06 $  &  $   2.9842   $   \\
   &$0.0125$&  $  7.7325e-06     $  &  $     2.1322   $  &  $   2.9761e-07     $ &  $    2.9405      $  &  $   5.5510e-07    $  &  $   2.9933    $   \\
\hline
\end{tabular}}\label{tex5.1t1}
\end{table}

\begin{table}[htbp]\label{table1}
\centering\small
\caption{The errors $e(\tau,h)$ and spatial convergence orders of the quasi-compact scheme (\ref{matappl}) by computing Example \ref{tex5.1} for different $\lambda$ with $j=5$. }\vspace{1em} {\begin{tabular} {@{}cccccccc@{}} \hline
 &  &\multicolumn{2}{c}{$\lambda=1$} & \multicolumn{2}{c}{$\lambda=10$ } & \multicolumn{2}{c}{$\lambda=50$ }\\
\cline{3-4} \cline{5-6}\cline{7-8}
$\alpha $&$h$ & $e(\tau,h)$ &rate & $e(\tau,h)$ &rate  & $e(\tau,h)$ &rate \\
 \hline
1.1&$0.1$   & $   6.0259e-06    $ &   $         $   & $   4.0133e-07    $ &   $    $  & $    1.4755e-09    $ &   $         $ \\
   &$0.05$  & $   7.6037e-07     $ &   $    2.9864    $   &  $    6.1913e-08    $ &   $    2.6965    $  &  $    1.9198e-09 $ &   $  -3.7982    $ \\
   &$0.025$ & $   9.5387e-08     $ & $     2.9948    $     &  $   8.3494e-09    $  &   $    2.8905   $  &  $   5.2808e-10     $  &   $  1.8622     $   \\
   &$0.0125$& $   1.1945e-08      $ &   $    2.9974     $     &  $   1.0774e-09    $  &   $   2.9541   $   &  $   8.6959e-11  $  &   $    2.6023  $  \\
 \hline
1.5&$0.1$   & $     9.1408e-05    $ &   $          $   & $   5.8760e-06  $ &   $          $ & $   1.1852e-02   $ &   $          $  \\
   &$0.05$  &  $   1.1772e-05        $ & $  2.9569    $   &  $ 9.3548e-07     $ &   $    2.6510  $  &  $    1.9228e-08   $ &   $   1.9234    $    \\
   &$0.025$ &  $    1.4927e-06   $  &  $   2.9794    $  &  $ 1.3004e-07    $  &  $   2.8467 $  &  $    5.6716e-09    $  &  $ 1.7614 $   \\
   &$0.0125$&  $     1.8791e-07     $  &   $   2.9898    $ &  $   1.7048e-08    $  &   $    2.9314     $  &  $     9.7853e-10   $  &   $   2.5351     $   \\
\hline

1.9&$0.1$   & $     2.7192e-04  $ &   $      $         &   $     1.7584e-05   $ &   $         $ &     $   1.3153e+47    $ &   $         $  \\
   &$0.05$  &  $   3.4977e-05      $ &   $  2.9587     $    &  $   2.7367e-06     $ &   $   2.6837  $  &  $    Inf     $ &   $ -Inf     $  \\
   &$0.025$ &  $   4.4201e-06      $  &  $  2.9842     $     &  $    3.7638e-07     $  &  $   2.8622  $  &  $  NaN     $  &  $     NaN   $   \\
   &$0.0125$&  $    5.5510e-07    $  &  $     2.9933    $  &  $   4.8967e-08       $ &  $    2.9423        $  &  $   NaN     $  &  $   NaN      $   \\
\hline
\end{tabular}}\label{tex4444}
\end{table}
\begin{example}\label{tex5.2}
Consider the following  tempered space fractional diffusion equation
\begin{equation}
\frac{\partial u}{\partial t}=\,_xD_1^{\alpha,\lambda }u(x)-e^{-t+\lambda x}\left((1-x)^j+\frac{\Gamma(j+1) (1-x)^{j -  \alpha}}{\Gamma(1+j - \alpha)}+\alpha\lambda^{\alpha-1} (\lambda (1-x)^j-j (1-x) ^{j-1})-\lambda^\alpha (1-x) ^j\right) ,
\end{equation}
where $ (x,t)\in(0,1)\times(0,0.1] $ and $j\in\mathbb{N}$.  The boundary conditions are $u(0,t)=e^{-t }$,  $u(1,t)=0 $ and  the initial value is $u(x,0)=e^{\lambda x}(1-x)^j, \,x\in[0,1]$. The exact solution is  $u(x)=e^{-t+\lambda x}(1-x)^{j}$.
\end{example}
\begin{table}[htbp]
\centering\small
\caption{ The errors $e(\tau,h)$ and spatial convergence orders of the quasi-compact scheme (\ref{matapp22}) by computing Example \ref{tex5.2} with $\lambda=1$.}\vspace{1em} {\begin{tabular} {@{}cccccccc@{}} \hline
 &  &\multicolumn{2}{c}{$j=1$} & \multicolumn{2}{c}{$j=3$ }& \multicolumn{2}{c}{$j=5$ }\\
\cline{3-4} \cline{5-6} \cline{7-8}
$\alpha$&$h$& $e(\tau,h)$ &rate & $e(\tau,h)$ &rate  & $e(\tau,h)$ &rate \\
 \hline
1.1&$0.1$   & $   1.2169e-04    $ &   $         $   & $    1.9583e-05    $ &   $         $  & $     1.6380e-05  $ &   $         $ \\
   &$0.05$  & $   3.9867e-05    $ &   $   1.6100   $   &  $     2.4327e-06   $ &   $   3.0090     $  &  $   2.0669e-06   $ &   $ 2.9864        $ \\
   &$0.025$ & $   1.0612e-05    $ & $   1.9095    $     &  $   3.0398e-07     $  &   $  3.0005   $  &  $    2.5929e-07   $  &   $     2.9948  $   \\
   &$0.0125$& $    2.1140e-06   $ &   $   2.3277    $     &  $  3.8018e-08  $  &   $ 2.9992    $   &  $   3.2469e-08   $  &   $    2.9974   $  \\
 \hline
1.5&$0.1$   & $  7.9411e-04     $ &   $          $   & $   6.0140e-05   $ &   $          $ & $    2.4847e-04  $ &   $          $  \\
   &$0.05$  &  $   1.6852e-04       $ & $  2.2364   $   &  $   9.4045e-06   $ &   $    2.6769  $  &  $     3.2000e-05  $ &   $   2.9569   $    \\
   &$0.025$ &  $    3.2337e-05   $  &  $   2.3816   $  &  $   1.3392e-06    $  &  $  2.8120    $  &  $   4.0577e-06     $  &  $  2.9794     $   \\
   &$0.0125$&  $    5.9527e-06     $  &   $  2.4416    $ &  $     1.8105e-07   $  &   $   2.8869    $  &  $   5.1080e-07  $  &   $    2.9898   $   \\
\hline
1.9&$0.1$   & $   1.7091e-03   $ &   $         $          &   $    3.0956e-04   $ &   $         $ &     $   7.3915e-04    $ &   $         $  \\
   &$0.05$  &  $  4.0278e-04      $ &   $  2.0852    $     &  $   4.5756e-05     $ &   $ 2.7582   $  &  $     9.5077e-05  $ &   $    2.9587 $  \\
   &$0.025$ &  $    9.2145e-05    $  &  $   2.1280    $     &  $    6.2106e-06  $  &  $  2.8812   $  &  $   1.2015e-05   $  &  $  2.9842    $   \\
   &$0.0125$&  $   2.1019e-05    $  &  $   2.1322     $    &  $     8.0898e-07    $ &  $    2.9405     $  &  $  1.5089e-06    $  &  $     2.9933  $   \\
\hline
\end{tabular}}\label{tex5.2t1}
\end{table}

\begin{table}[htbp]
\centering\small
\caption{  The errors $e(\tau,h)$ and spatial convergence orders of the quasi-compact scheme (\ref{matapp22}) by computing   Example \ref{tex5.2} for different $\lambda $ with $j=5$.}\vspace{1em} {\begin{tabular} {@{}cccccccc@{}} \hline
 &  &\multicolumn{2}{c}{$\lambda=1$} & \multicolumn{2}{c}{$\lambda=10$ } & \multicolumn{2}{c}{$\lambda=50$ }\\
\cline{3-4} \cline{5-6}\cline{7-8}
$\alpha $&$h$ & $e(\tau,h)$ &rate & $e(\tau,h)$ &rate  & $e(\tau,h)$ &rate \\
 \hline
1.1&$0.1$   & $    1.6380e-05      $ &   $         $   & $  8.8399e-03   $ &   $               $  & $   7.6498e+12   $ &   $         $ \\
   &$0.05$  & $    2.0669e-06     $ &   $      2.9864  $   &  $   1.3637e-03      $ &   $  2.6965       $  &  $   9.9538e+12   $ &   $    -3.7982   $ \\
   &$0.025$ & $   2.5929e-07       $ & $   2.9948    $     &  $     1.8391e-04   $  &   $   2.8905   $  &  $    2.7379e+12    $  &   $    1.8622     $   \\
   &$0.0125$& $     3.2469e-08     $ &   $     2.9974     $     &  $    2.3731e-05  $  &   $  2.9541    $   &  $  4.5086e+11    $  &   $   2.6023      $  \\
 \hline
1.5&$0.1$   & $   2.4847e-04        $ &   $          $   & $    1.2943e-01    $ &   $          $ & $   6.1451e+19     $ &   $          $  \\
   &$0.05$  &  $ 3.2000e-05           $ & $    2.9569   $   &  $  2.0605e-02   $ &   $   2.6510  $  &  $   9.9692e+13    $ &   $    19.234    $    \\
   &$0.025$ &  $    4.0577e-06     $  &  $     2.9794   $  &  $  2.8644e-03    $  &  $   2.8467    $  &  $ 2.9406e+13         $  &  $ 1.7614      $   \\
   &$0.0125$&  $    5.1080e-07        $  &   $   2.9898    $ &  $  3.7550e-04       $  &   $   2.9314      $  &  $  5.0734e+12    $  &   $    2.5351   $   \\
\hline

1.9&$0.1$   & $   7.3915e-04    $ &   $      $         &   $     3.8731e-01    $ &   $         $ &     $     6.8193 e+68    $ &   $         $  \\
   &$0.05$  &  $     9.5077e-05     $ &   $   2.9587     $    &  $  6.0281e-02        $ &   $ 2.6837     $  &  $      NaN     $ &   $  NaN     $  \\
   &$0.025$ &  $    1.2015e-05    $  &  $     2.9842   $     &  $  8.2903e-03       $  &  $  2.8622     $  &  $    NaN   $  &  $     NaN   $   \\
   &$0.0125$&  $    1.5089e-06    $  &  $     2.9933    $  &  $  1.0786e-03        $ &  $      2.9423      $  &  $  NaN      $  &  $    NaN      $   \\
\hline
\end{tabular}}\label{tex5555}
\end{table}

Table \ref{tex5.2t1} confirms the convergence orders of the scheme of the right tempered fractional diffusion equation and shows the required regularity for the to be approximated solution. And Table \ref{tex5555} verifies the required stability condition $\lambda h<1$.


\begin{example}
The following two dimensional   fractional diffusion problem
\begin{equation}\label{TE5.3}
\frac{\partial u(x,y,t)}{\partial t}=\,_0D_x^{\alpha ,\lambda_1 }u(x,y,t) +\,_0D_y^{\beta,\lambda_2  }u(x,y,t) +f(x,y,t)
\end{equation}
is considered   in the domain $ \Omega= (0,1)^2$ and $t\in(0,1]$.
The source term is
$$
\begin{array}{lll}
\displaystyle
f(x,t)
=-e^{-t-\lambda_1 x-\lambda_2 y}
 \left[\left(x^{4}+\frac{\Gamma(5) x^{4 -  \alpha}}{\Gamma(5 - \alpha)}-\alpha\lambda_1^{\alpha-1} (4 x ^{3}-\lambda_1 x^{4})-\lambda_1^\alpha x ^{4}\right.\right.\\\\
\displaystyle~~~~    ~~~~   ~~~~  ~~~~  ~~~~   ~~~~  ~~~~  ~~~~ ~~~~
  \left. \left.    -\left(  x^5+\frac{\Gamma(6) x^{5 -  \alpha}}{\Gamma(6 - \alpha)}-\alpha\lambda_1^{\alpha-1} (5 x ^{4}-\lambda_1 x^5)-\lambda_1^\alpha x ^5 \right)\right)y^{4}(1-y)\right. +\\\\
\displaystyle~~~~    ~~~~   ~~~~  ~~~~  ~~~~   ~~~~  ~~~~  ~~~~
\left. \left( \frac{\Gamma(5) y^{4 - \beta}}{\Gamma(5 -\beta)}-\beta\lambda_2^{\beta-1} (4 y^{3}-\lambda_2 y^4)-\lambda_2^\beta y ^4
 -\left(   \frac{\Gamma(6) y^{5 - \beta}}{\Gamma(6 - \beta)}-\beta\lambda_2^{\beta-1} (5 y ^{4}-\lambda_2 y^5)-\lambda_2^\beta y^5 \right)\right)x ^{4}(1-x)\right].
\end{array}
$$
The exact solution is given by  $u(x,y,t)= e^{-t-\lambda_1 x-\lambda_2 y}x^{4} (1-x) y^{4} (1-y)$. The boundary conditions are $u(0,y,t)=u(x,0,t)=u(x,1,t)=u(1,y,t)=0$ with $(x,y)\in\partial \Omega$ and $t\in[0,1]$.  The initial value is  $u(x,y,0)=e^{-\lambda_1 x-\lambda_2 y}x^{4} (1-x) y^{4} (1-y)$ with $ (x,y)\in[0,1]^2$.
\end{example}

To solve (\ref{TE5.3}), we derive the quasi-compact D'yakonov ADI scheme in matrix form as
 \begin{equation*}
    \begin{array}{lll}
\displaystyle \left( B_l^{ \alpha } -\frac{\tau}{2}P^{ \alpha }_l\right) U^{*}=\left(B_l^{ \alpha } +\frac{\tau}{2}P^{ \alpha }_l\right) U^{n}\left(B_l^{ \beta }+ \frac{\tau}{2}P^{ \beta }_l\right)^{T} + \tau B_l^{ \alpha } f^{n+1/2} \left(B_l^{\beta }\right)^{T},\\\\
\displaystyle U^{n+1}\left(B_l^{\beta }-\frac{\tau}{2}P^{ \beta }_l\right)^{T}=U^{*}.
 \end{array}
   \end{equation*}

In this example, we take the uniform meshes with the space step size $h_y=h_x$  and the time step size $\tau=h^\frac{3}{2}$.
From Table \ref{tt5.3}, it can be seen that the numerical results are stable and convergent, and the third order accuracy in space is verified.

\begin{table}[htbp]
\centering\small
\caption{The convergence orders of the quasi-compact D'yakonov ADI scheme by computing (\ref{TE5.3})  at $t=1$ with $\tau=h^\frac{3}{2}$ and $\lambda_1=\lambda_2=\lambda$.} {\begin{tabular} {@{}cccccccc@{}} \hline
  & &  \multicolumn{2}{c}{$(\alpha,\beta)=(1.2,1.5)\,\&\,\lambda=0.1$} & \multicolumn{2}{c}{$(\alpha,\beta)=(1.5,1.9)\,\&\,\lambda=0.1$} & \multicolumn{2}{c}{$(\alpha,\beta)=(1.2,1.5)\,\&\,\lambda=10$} \\
\cline{3-4} \cline{5-6} \cline{7-8}
D'yakonov &$h_x(h_y=h_x)$ & $e(\tau,h)$ &rate & $e(\tau,h)$ &rate & $e(\tau,h)$ &rate \\
 \hline
&$0.1$ &$  6.7629e-06   $ & $  $      & $   9.4666e-06 $ & $   $     & $   3.6415e-10  $ &   $ $ \\
         &$0.05 $& $ 8.4125e-07  $ & $ 3.0070 $& $   1.2061e-06 $ & $2.9725 $ &  $  6.1485e-11 $ &   $   2.5662   $ \\
         &$ 0.025$ & $ 1.0543e-07$ & $ 2.9962 $& $  1.5279e-07 $  & $2.9807 $ & $ 9.6300e-12  $  &    $     2.6746   $ \\
         &$ 0.0125$ & $1.3164e-08$ & $ 3.0017 $& $  1.9169e-08 $  & $2.9947 $ &  $ 1.3539e-12 $  &    $    2.8304 $  \\
\hline
\end{tabular}}\label{tt5.3}
\end{table}

\begin{example}\label{eq5.3}
The following   two sided fractional diffusion problem
\begin{equation}\label{ex333}
\frac{\partial u(x, t)}{\partial t}=\,_0D_x^{\alpha ,\lambda  }u(x,t) +\,_xD_{1}^{\alpha,\lambda }u(x,t) +f(x,t)
\end{equation}
is considered in the domain $\Omega=(0,1)$ and $t\in(0,1]$. The exact solution is given by  $u(x,t)= e^{-t-\lambda x}x^{4}(1-x)^{4}$. The boundary conditions are $u(0,t)=u(1,t)=0$ with $t\in[0,1]$.
 The initial value is  $u(x,0)=e^{-\lambda x}x^{4}(1-x)^{4}$.

To get the source term, we need to calculate the right fractional derivative $\,_xD_{1}^{(\alpha,\lambda )}u(x,t)$ of the exact solution and obtain that
$$
\begin{array}{lll}
\displaystyle
\,_xD_{1}^{(\alpha,\lambda )}e^{-\lambda x}x^{4}(1-x)^{4}\\\\
\displaystyle
=e^{ \lambda x}\,_xD_{1}^{ \alpha }e^{-2\lambda x}x^{4}(1-x)^{4}\\\\
\displaystyle
=e^{\lambda (x-2)}\,_xD_{1}^{\alpha  }\left(e^{2\lambda(1-x)}\sum_{m=0}^{4}((-1)^m\left(\begin{array}{c}4\\m\end{array}\right)(1-x)^m) \right)
\\\\
\displaystyle
=e^{\lambda (x-2)}\,_xD_{1}^{\alpha  }\left(\sum_{j=0}^{\infty}\frac{(2\lambda(1-x))^j}{j!} \sum_{m=0}^{4}((-1)^m\left(\begin{array}{c}4\\m\end{array}\right)(1-x)^m) )\right)\\\\
\displaystyle
=e^{\lambda (x-2)}\sum_{j=0}^{\infty} \frac{(2\lambda )^j}{j!}\sum_{m=0}^{4}\left((-1)^m\left(\begin{array}{c}4\\m\end{array}\right) \frac{\Gamma(5+m+j) }{\Gamma(5+m+j-\alpha)}(1-x)^{j+ 4+m-\alpha}    \right)
 \\\\
\displaystyle
\simeq e^{\lambda (x-2)}\sum_{j=0}^{50} \frac{(2\lambda )^j}{j!}\sum_{m=0}^{4}\left((-1)^m\left(\begin{array}{c}4\\m\end{array}\right) \frac{\Gamma(5+m+j) }{\Gamma(5+m+j-\alpha)}(1-x)^{j+ 4+m-\alpha}    \right).
\end{array}
$$

Then the source term
$$
\begin{array}{lll}
\displaystyle
f(x,t)\simeq-e^{-t }\left[e^{ -\lambda x}\left(x^4(1-x)^{4}+ \sum_{m=0}^{4}\left((-1)^m\left(\begin{array}{c}4\\m\end{array}\right) \frac{\Gamma(5+m) }{\Gamma(5+m-\alpha)}x^{ 4+m-\alpha}    \right)  
-2\lambda^\alpha x^{4}(1-x)^{4}  \right) \right.
\\\\\displaystyle~~~~ ~~~~  ~~~~ ~~~~
 \left. +e^{\lambda (x-2)}\sum_{j=0}^{50} \frac{(2\lambda )^j}{j!}\sum_{m=0}^{4}\left((-1)^m\left(\begin{array}{c}4\\m\end{array}\right) \frac{\Gamma(5+m+j) }{\Gamma(5+m+j-\alpha)}(1-x)^{j+ 4+m-\alpha} \right)   \right] .
\end{array}
$$

\end{example}

Since the compact difference operator  $_L\mathcal{B}$ commutes with $_R\mathcal{B}$,
to solve (\ref{ex333}), we introduce the operator splitting method to derive the numerical schemes and the matrix forms of the schemes are as follows:
 \begin{equation*}
    \begin{array}{lll}
\displaystyle   B_l^{ \alpha } U^{*}=(B_l^{ \alpha } + \tau P^{ \alpha }_l) U^{n} + \frac{\tau}{2} B_l^{ \alpha } f^{n+1/2} ,\\\\
\displaystyle (B_r^{ \alpha }- \tau P^{\alpha}_r)U^{n+1} =  B_r^{ \alpha } U^{*}+ \frac{\tau}{2} B_r^{ \alpha } f^{n+1/2}.
 \end{array}
   \end{equation*}

Table \ref{exbb33} shows the numerical results obtained by using the quasi-compact operator splitting method to solve  (\ref{ex333}) with $\tau=h^3$. It can be noted that the convergence orders are three in spacial direction which confirms the theoretical estimations.

\begin{table}[htbp]
\centering\small
\caption{The errors $e(\tau,h)$ and spatial convergence orders of the operator splitting scheme by computing Example \ref{eq5.3} with $\lambda=0.1$.} {\begin{tabular} {@{} ccccccc@{}} \hline
    &  \multicolumn{2}{c}{$ \alpha = 1.2 $} & \multicolumn{2}{c}{$ \alpha = 1.5 $} & \multicolumn{2}{c}{$ \alpha =1.8$} \\
\cline{2-3} \cline{4-5} \cline{6-7}
   $h$ & $e(\tau,h)$ &rate & $e(\tau,h)$ &rate & $e(\tau,h)$ &rate \\
 \hline
          $0.1$ &$     4.0467e-06  $ & $  $      & $   5.8747e-06   $ & $   $     & $    8.4320e-06   $ &   $ $ \\
          $0.05 $& $   5.6875e-07  $ & $  2.8309  $& $  7.8215e-07    $ & $ 2.9090  $ &  $   9.7711e-07  $ &   $   3.1093     $ \\
          $ 0.025$ & $  7.5560e-08 $ & $  2.9121  $& $  9.8146e-08   $  & $  2.9944 $ & $  1.0509e-07   $  &    $    3.2168      $ \\
          $ 0.0125$ & $ 9.7675e-09 $ & $   2.9516 $& $ 1.2211e-08    $  & $  3.0067 $ &  $ 1.1657e-08   $  &    $    3.1724   $  \\
\hline
\end{tabular}}\label{exbb33}
\end{table}

\section{Conclusions}
Sometimes the tempered power-law diffusions, instead of pure power-law diffusions, are more physical/reasonable choice in practical applications. This paper focuses on providing the quasi-compact schemes for the tempered fractional diffusion equations. Not only its derivation but the proof of its numerical stability and convergence are much different from the ones of the fractional diffusion equations. The detailed theoretical results are presented, and some techniques are introduced in the analysis. Extensive numerical simulations are performed to show the effectiveness of the schemes, and the third order convergence is confirmed.

 \appendix
 \renewcommand{\appendixname}{Appendix~ }
  \section{Some Lemmas}
Define
\begin{equation} \label{dfid}
 \begin{array}{l }
\displaystyle
\,_a I_x^{(p,\lambda)}u(x)=\frac{e^{-\lambda x}}{\Gamma(p)} \int^{x}_{a}(x-s)^{ p-1 }e^{\lambda s}u(s) ds,\\\\
\displaystyle
\,_x I_{b}^{(p,\lambda)}u(x)=\frac{ e^{\lambda x}}{\Gamma(p)}  \int^{b}_{x}(s -x)^{ p-1 }e^{-\lambda s}u(s) ds ,\\\\
\displaystyle
\,_a D_x^{(p,\lambda)}u(x)=\frac{e^{-\lambda x}}{\Gamma(m-p)}\frac{d^m}{dx^m} \int^{x}_{a}(x-s)^{m-p-1 }e^{\lambda s}u(s) ds,\\\
\displaystyle
\,_x D_{b}^{(p,\lambda)}u(x)=\frac{(-1)^m e^{\lambda x}}{\Gamma(m-p)}\frac{d^m}{dx^m} \int^{b}_{x}(s -x)^{m-p-1 }e^{-\lambda s}u(s) ds,
\end{array}
\end{equation}
where $p$ is a positive constant  and  $m-1<p<m$.
Now we show some properties of the tempered fractional calculus.

\begin{lemma}
Let $u(x)$ be continuous on $[a,b]$, $ p,q>0$ and $ \lambda>0$. Then the integration of arbitrary real order has the properties:
\begin{equation}\label{la1}
  \,_a I_x^{(p,\lambda)}\left(\,_a I_x^{(q,\lambda)}u(x)\right)=\,_a I_x^{(p+q,\lambda)}u(x)=\,_a I_x^{(q,\lambda)}\left(\,_a I_x^{(p,\lambda)}u(x)\right)
\end{equation}
and
\begin{equation}\label{la2}
   \,_x I_{b}^{(p,\lambda)}\left(\,_x I_{b}^{(q,\lambda)}u(x)\right)=\,_x I_{b}^{(p+q,\lambda)}u(x)=\,_x I_{b}^{(q,\lambda)}\left(\,_x I_{b}^{(p,\lambda)}u(x)\right).
\end{equation}
\end{lemma}
\begin{proof}
Taking into account the definition of the integral  $\,_x I_{b}^{(p,\lambda)}u(x)$, we have
\begin{equation*}
 \begin{array}{l }
\displaystyle
  \,_x I_{b}^{(p,\lambda)}\left(\,_x I_{b}^{(q,\lambda)}u(x)\right)
  =\frac{ e^{\lambda x}}{\Gamma(p)}  \int^{b}_{x}(s -x)^{ p-1 }e^{-\lambda s}\,_s I_{b}^{(q,\lambda)}u(s) ds\\\\
 \displaystyle~~~~ ~~~~ ~~~~ ~~~~ ~~~~ ~~~~ ~~~~~
  =\frac{ e^{\lambda x}}{\Gamma(p)}  \int^{b}_{x}(s -x)^{ p-1 }e^{-\lambda s} \frac{ e^{\lambda s}}{\Gamma(q)}  \int^{b}_{s}(\eta -x)^{ q-1 }e^{-\lambda \eta}u(\eta) d\eta ds\\\\
\displaystyle~~~~ ~~~~ ~~~~ ~~~~ ~~~~ ~~~~ ~~~~~
  =\frac{ e^{\lambda x}}{\Gamma(p)\Gamma(q)}  \int^{b}_{x} e^{-\lambda \eta}u(\eta) d\eta   \int^{\eta}_{x}(s -x)^{ p-1 }(\eta -s)^{ q-1 }ds.
\end{array}
\end{equation*}
Here we use the substitution $s=x+\zeta(\eta-x)$ to evaluate the integration from $x$ to $\eta$, and obtain
\begin{equation*}
 \begin{array}{l }
\displaystyle
    \int^{\eta}_{x}(s -x)^{ p-1 }(\eta -s)^{ p-1 }ds
    = \int^{1}_{0}(\zeta(\eta-x))^{ p-1 }((1-\zeta)(\eta-x))^{ p-1 } (\eta-x) d\zeta
   =(\eta-x) ^{ p+q-1 }B(p,q) \\\\
 \displaystyle ~~~~  ~~~~ ~~~~ ~~~~ ~~~~ ~~~~ ~~~~  ~~~~ ~~~~ ~~~~~
   =(\eta-x) ^{ p+q-1 }\frac{\Gamma(p)\Gamma(p) }{\Gamma(p+q)},
\end{array}
\end{equation*}
where $B(p,q)$ is the Beta function. Therefore,
\begin{equation*}
 \begin{array}{l }
\displaystyle
  \,_x I_{b}^{(p,\lambda)}\left(\,_x I_{b}^{(q,\lambda)}u(x)\right)
  =\frac{ e^{\lambda x}}{\Gamma(p+q) }  \int^{b}_{x} e^{-\lambda \eta}(\eta-x) ^{ p+q-1 }u(\eta) d\eta \\\\
\displaystyle~~~~ ~~~~ ~~~~ ~~~~ ~~~~ ~~~~ ~~~~ ~
  =\,_a I_x^{(p+q,\lambda)}u(x).
\end{array}
\end{equation*}
Obviously, $p$ and $q$ can be interchanged. Then
\begin{equation*}
  \,_x I_{b}^{(p,\lambda)}\left(\,_x I_{b}^{(q,\lambda)}u(x)\right)=\,_x I_{b}^{(p+q,\lambda)}u(x)=\,_x I_{b}^{(q,\lambda)}\left(\,_x I_{b}^{(p,\lambda)}u(x)\right).
\end{equation*}
The proof of (\ref{la1}) is omitted here.
\end{proof}

\begin{lemma}\label{lemma3}
Let $k, m$ be positive integers, $ m-1<p <m $, $ \lambda>0$ and $u(x)\in C^{m+k-1}[a,b]$.
Then we have
 \begin{enumerate}
  \item  $\,_a D_x^{(k,\lambda)}\left(\,_a D_x^{(p,\lambda)}u(x)\right)=\,_a D_x^{(k+p,\lambda)}u(x) $,
  \item  $\,_x D_{b}^{(k,\lambda)}\left(\,_x D_{b}^{(p,\lambda)}u(x)\right)=\,_x D_{b}^{(k+p,\lambda)}u(x)$,
  \item  $ \,_a D_x^{(p,\lambda)}\left(\,_a D_x^{(k,\lambda)}u(x)\right)=\,_a D_x^{(k+p,\lambda)}u(x)- \sum\limits_{j=0}^{k-1}\frac{(e^{\lambda a}u(a))^{(j)}e^{-\lambda x}}{\Gamma(j+1-p-k)}(x-a)^{j-p-k}$,
  \item  $ \,_x D_{b}^{(p,\lambda)}\left(\,_x D_{b}^{(k,\lambda)}u(x)\right)=\,_x D_{b}^{(k+p,\lambda)}u(x) -\sum\limits_{j=0}^{k-1}\frac{(-1)^{k-j}\left(e^{-\lambda b}u(b)\right)^{(j)}e^{ \lambda x}}{\Gamma(j+1-p-k)}(b-x)^{j-p-k}$.
\end{enumerate}
\end{lemma}
\begin{proof}
The proofs of 1 and 2:

 From (\ref{dfid}), we get
\begin{equation*}
 \begin{array}{l }
\displaystyle
 \,_a D_x^{(k,\lambda)}\left(\,_a D_x^{(p,\lambda)}u(x)\right)
 =  e^{-\lambda x} \frac{d^k}{dx^k} \left( e^{\lambda x}\,_a D_x^{(p,\lambda)}u(x)\right)
 \\\\
\displaystyle~~~~ ~~~~ ~~~~ ~~~~ ~~~~ ~~~~ ~~~~ ~~~~
 =  e^{-\lambda x} \frac{d^k}{dx^k} \left ( e^{\lambda x}\frac{e^{-\lambda x}}{\Gamma(m-p)}\frac{d^m}{dx^m} \int^{x}_{a}(x-s)^{m-p-1 }e^{\lambda s}u(s) ds\right)
  \\\\
\displaystyle~~~~ ~~~~ ~~~~ ~~~~ ~~~~ ~~~~ ~~~~ ~~~~
 = \frac{ e^{-\lambda x}}{ \Gamma(m-p)} \frac{d^k}{dx^k} \left(  \frac{d^m}{dx^m} \int^{x}_{a}(x-s)^{m-p-1 }e^{\lambda s}u(s) ds \right)
  \\\\
\displaystyle~~~~ ~~~~ ~~~~ ~~~~ ~~~~ ~~~~ ~~~~ ~~~~
 = \frac{ e^{-\lambda x}}{ \Gamma(m+k-p-k)} \frac{d^{k+m}}{dx^{k+m}} \left(  \int^{x}_{a}(x-s)^{m+k-p-k-1 }e^{\lambda s}u(s) ds \right)
 \\\\
\displaystyle~~~~ ~~~~ ~~~~ ~~~~ ~~~~ ~~~~ ~~~~  ~~~~
 =  \,_a D_x^{(k+p,\lambda)}u(x)
\end{array}
\end{equation*}
and
\begin{equation*}
 \begin{array}{l }
\displaystyle
\,_x D_{b}^{(k,\lambda)}\left(\,_x D_{b}^{(p,\lambda)}u(x)\right)
 = (-1)^k e^{ \lambda x} \frac{d^k}{dx^k} \left( e^{-\lambda x}\frac{(-1)^me^{\lambda x}}{\Gamma(m-p)}\frac{d^m}{dx^m} \int^{x}_{a}(s-x)^{m-p-1 }e^{-\lambda s}u(s) ds\right)
  \\\\
\displaystyle~~~~ ~~~~ ~~~~ ~~~~ ~~~~ ~~~~ ~~~~ ~~~~
 = \frac{(-1)^{k+m} e^{ \lambda x}}{\Gamma(m-p)} \frac{d^{k+m}}{dx^{k+m}} \left(  \int^{x}_{a}(s-x)^{m-p-1 }e^{-\lambda s}u(s) ds\right)
 \\\\
\displaystyle~~~~ ~~~~ ~~~~ ~~~~ ~~~~ ~~~~ ~~~~  ~~~~
 =\,_x D_{b}^{(k+p,\lambda)}u(x).
 \end{array}
\end{equation*}

The proof of 3:

To consider the operator $ \,_a D_x^{(p,\lambda)}\left(\,_a D_x^{(k,\lambda)}u(x)\right)$, we must mention that
\begin{equation}\label{lemma31}
 \begin{array}{l }
\displaystyle
  \,_a I_{x}^{(k,\lambda)}\left(\,_a D_{x}^{(k,\lambda)}u(x)\right)
  =\frac{ e^{-\lambda x}}{ (k-1)!}  \int^{x}_{a}(x -s)^{ k-1 }e^{ \lambda s}\,_a D_{s}^{(k,\lambda)}u(s) ds\\\\
 \displaystyle~~~~ ~~~~ ~~~~ ~~~~ ~~~~ ~~~~ ~~~~ ~~
  =\frac{ e^{-\lambda x}}{ (k-1)!}  \int^{x}_{a}(x -s)^{ k-1 } \frac{d^k}{ds^k} e^{ \lambda s}u(s)  ds\\\\
\displaystyle~~~~ ~~~~ ~~~~ ~~~~ ~~~~ ~~~~ ~~~~ ~~
  =u(x)- \sum\limits_{j=0}^{k-1}\frac{(e^{\lambda a}u(a))^{(j)}e^{-\lambda x}(x-a)^{j }}{\Gamma(j+1 )} ,
\end{array}
\end{equation}
where $(e^{\lambda a}u(a))^{(j)}$ is the differentiation of order $j$ of the function $e^{\lambda x}u(x)$ at point $x=a$; 
\begin{equation}\label{lemma32}
 \begin{array}{l }
\displaystyle
   \,_a D_{x}^{(p,\lambda)}u(x)
  = \frac{e^{-\lambda x}}{\Gamma(m-p)}\frac{d^m}{dx^m} \int^{x}_{a}(x-s)^{m-p-1 }e^{\lambda s}u(s) ds\\\\
 \displaystyle~~~~ ~~~~ ~~~~ ~~~~ ~~
  =e^{-\lambda x}\frac{d^m}{dx^m}e^{ \lambda x} \frac{e^{-\lambda x}}{\Gamma(m-p)} \int^{x}_{a}(x-s)^{m-p-1 }e^{\lambda s}u(s) ds\\\\
 \displaystyle~~~~ ~~~~ ~~~~ ~~~~ ~~
  =  \,_a D_{x}^{(m,\lambda)}\left(\,_a I_{x}^{(m-p,\lambda)}u(x)\right)
\end{array}
\end{equation}
and
\begin{equation}\label{lemma33}
   \,_a D_{x}^{(p+k,\lambda)}\left(\,_a I_{x}^{(k,\lambda)}u(x)\right)
  = \,_a D_{x}^{(m+k,\lambda)} \,_a I_{x}^{(m-p,\lambda)}\left(\,_a I_{x}^{(k,\lambda)}u(x)\right)
  = \,_a D_{x}^{(m+k,\lambda)}\left( \,_a I_{x}^{(k+m-p,\lambda)} u(x)\right)= \,_a D_{x}^{(p,\lambda)}  u(x).
\end{equation}
Using (\ref{lemma31}), (\ref{lemma32}),  (\ref{lemma33}) and
\[  \,_a D_{x}^{(p ,\lambda)}(e^{-\lambda x}(x-a)^{j })=\frac{\Gamma(1+j)e^{-\lambda x}}{\Gamma(1+j-p)}(x-a)^{j-p },\]
we obtain
\begin{equation*}
 \begin{array}{l }
\displaystyle
  \,_a D_x^{(p,\lambda)}\left(\,_a D_x^{(k,\lambda)}u(x)\right)
  = \,_a D_{x}^{(p+k,\lambda)}\left(\,_a I_{x}^{(k,\lambda)} \,_a D_{x}^{( k,\lambda)}u(x)\right)\\\\
 \displaystyle~~~~ ~~~~ ~~~~ ~~~~ ~~~~ ~~~~ ~~~~ ~~~~
  = \,_a D_{x}^{(p+k,\lambda)}\left(u(x)- \sum\limits_{j=0}^{k-1}\frac{(e^{\lambda a}u(a))^{(j)}e^{-\lambda x}(x-a)^{j }}{\Gamma(j+1 )}\right)\\\\
 \displaystyle~~~~ ~~~~ ~~~~ ~~~~ ~~~~   ~~~~ ~~~~  ~~~~
  =\,_a D_x^{(k+p,\lambda)}u(x)- \sum\limits_{j=0}^{k-1}\frac{(e^{\lambda a}u(a))^{(j)}e^{-\lambda x}}{\Gamma(j+1-p-k)}(x-a)^{j-p-k}.
\end{array}
\end{equation*}

The proof of 4:

Since
\begin{equation*}
 \begin{array}{l }
\displaystyle
  \,_x I_{b}^{(k,\lambda)}\left(\,_x D_{b}^{(k,\lambda)}u(x)\right)
  =\frac{(-1)^k e^{\lambda x}}{ (k-1)!}  \int^{x}_{a}( s-x)^{ k-1 } \frac{d^k}{ds^k}(e^{- \lambda s} u(s)) ds\\\\
\displaystyle~~~~ ~~~~ ~~~~ ~~~~ ~~~~ ~~~~ ~~~~ ~~
  =u(x)- \sum\limits_{j=0}^{k-1}\frac{(-1)^{k-j}(e^{-\lambda b}u(b))^{(j)}e^{ \lambda x}(b-x)^{j }}{\Gamma(j+1 )} ,
\end{array}
\end{equation*}
\begin{equation*}
 \begin{array}{l }
\displaystyle
   \,_x D_{b}^{(p,\lambda)}u(x)
  =  \,_x D_{b}^{(m,\lambda)}\left(\,_x I_{b}^{(m-p,\lambda)}u(x)\right), \quad  \,_x D_{b}^{(p+k,\lambda)}\left(\,_x I_{b}^{(k,\lambda)}u(x)\right)
 = \,_x D_{b}^{(p,\lambda)}  u(x)
\end{array}
\end{equation*}
and
\[  \,_x D_{b}^{(p ,\lambda)}(e^{ \lambda x}(b-x)^{j })=\frac{\Gamma(1+j)e^{ \lambda x}}{\Gamma(1+j-p)}(b-x)^{j-p },\]
we obtain
\begin{equation*}
 \begin{array}{l }
\displaystyle
  \,_x D_b^{(p,\lambda)}\left(\,_x D_b^{(k,\lambda)}u(x)\right)
  = \,_x D_{b}^{(p+k,\lambda)}\left(\,_x I_{b}^{(k,\lambda)} \,_x D_{b}^{( k,\lambda)}u(x)\right)\\\\
 \displaystyle~~~~ ~~~~ ~~~~ ~~~~ ~~~~ ~~~~ ~~~~ ~~~~
  = \,_x D_{b}^{(p+k,\lambda)}(u(x)- \sum\limits_{j=0}^{k-1}\frac{(-1)^{k-j}(e^{-\lambda b}u(b))^{(j)}e^{ \lambda x}(b-x)^{j }}{\Gamma(j+1 )})\\\\
 \displaystyle~~~~ ~~~~ ~~~~ ~~~~ ~~~~   ~~~~ ~~~~  ~~~~
  =\,_x D_b^{(k+p,\lambda)}u(x)- \sum\limits_{j=0}^{k-1}\frac{(-1)^{k-j}(e^{-\lambda b}u(b))^{(j)}e^{ \lambda x}}{\Gamma(j+1-p-k)}(b-x)^{j-p-k}.
\end{array}
\end{equation*}

\end{proof}
\begin{lemma}
Let $k, m$ be positive integers, $ m-1<p <m $, $ \lambda>0$ and $u(x)$ be $(m+k-1)$-times continuously differentiable  on $(-\infty,+\infty)$. Then we have
 \begin{enumerate}
  \item  $\,_{-\infty} D_x^{(k,\lambda)}\left(\,_{-\infty} D_x^{(p,\lambda)}u(x)\right)=\,_{-\infty} D_x^{(k+p,\lambda)}u(x) $,
  \item  $\,_x D_{+\infty}^{(k,\lambda)}\left(\,_x D_{+\infty}^{(p,\lambda)}u(x)\right)=\,_x D_{+\infty}^{(k+p,\lambda)}u(x)$.
\end{enumerate}
\end{lemma}

\bibliographystyle{elsarticle-num}
\bibliography{<your-bib-database>}

\end{document}